\documentclass[12pt,reqno,twoside]{amsart}
\usepackage{graphicx}
\usepackage{amsmath}
\usepackage{amssymb}
  \newcommand{\absolute}[1] {\left|{#1}\right|}
  \newcommand{\grad}{\nabla}
\usepackage{color}

\newcommand{\Q}{{\mathbb {Q}}}

\newcommand{\R}{{\mathbb{R}}} 
\newcommand{\Z}{{\mathbb{Z}}}
\newcommand{\C}{{\mathbb{C}}}

\newcommand{\N}{{\mathbb{N}}}

\newcommand\RR{\mathbb R}

\renewcommand\Im{\operatorname{Im}}

\newcommand\vol{\operatorname{vol}}

\newcommand\abs[1]{\left|#1\right|}

\newcommand\inn[1]{\left\langle #1 \right\rangle}
\newcommand\set[1]{\left\{{#1}\right\}}

\newtheorem{thm}{Theorem}[section]
\newtheorem{lem}[thm]{Lemma}
\newtheorem{prop}[thm]{Proposition}

\numberwithin{equation}{section}

\begin{document}

\title[Lattice points on the Heisenberg group]{The lattice point counting problem \\ on the Heisenberg groups}






\author{Rahul Garg$^\ast$}
\address{Department of Mathematics, Technion, Haifa, Israel}
\email{rgarg@tx.technion.ac.il}
\thanks{$^\ast$ Supported by ISF Grant and PBC postdoctoral fellowship of CHE, Israel}

\author{Amos Nevo$^\dagger$}
\address{Department of Mathematics, Technion, Haifa, Israel}
\email{amosnevo6@gmail.com}
\thanks{$^\dagger$ Supported by ISF Grant}

\author{Krystal Taylor$^\sharp$}
\address{Institute of Mathematics and its Applications, Minneapolis, Minnesota}
\email{krystaltaylormath@gmail.com}
\thanks{$^\sharp$ Supported by ISF Grant, Israel and postdoctoral fellowship, U.S.A.}

\begin{abstract}
We consider the  radial and Heisenberg-homogeneous norms on the Heisenberg groups given by 
$N_{\alpha,A}((z,t)) = \left(\abs{z}^\alpha + A \abs{t}^{\alpha/2}\right)^{1/\alpha}$, for 
$\alpha \ge 2$ and $A>0$. This natural family includes the canonical Cygan-Kor\'anyi norm, corresponding to 
$\alpha =4$. We study the lattice points counting problem on the Heisenberg groups, namely establish 
an error estimate for the number of points that the lattice of integral points has in a ball of large 
radius $R$. The exponent we establish for the error in the case $\alpha=2$ is the best possible, 
in all dimensions.   
\end{abstract}

\maketitle
{\small \tableofcontents}

\section{Introduction, notation and statement of results}

\subsection{Euclidean and non-Euclidean lattice point counting problem}
The classical lattice point counting problem in Euclidean space considers a fixed compact convex set 
$B\subset \R^n$ with $0\in B$ an interior point, and  aims to establish an asymptotic of the form 
\begin{equation}\label{eq:def-euclid}
\abs{\Z^n \cap tB} = t^n \vol(B) + O_{\theta^\prime} \left(t^{n-\theta^\prime}\right) 
= \vol(tB) + O_{\kappa^\prime} \left(\left(\vol(tB)\right)^{\kappa^\prime}\right)
\end{equation}
with  $\theta^\prime > 0$ as large as possible (or $\kappa^\prime < 1$ as small as possible) for 
large parameter $t$. We let $\theta$ denote the supremum of $\theta^\prime$ that are admissible in 
(\ref{eq:def-euclid}) (and $\kappa$ the infimum of admissible $\kappa^\prime$).

This problem has a long history, and arises naturally in many applications. Among those, we mention 
just the following two.
\begin{itemize} 
\item The problem of obtaining the asymptotics of the Laplace eigenvalue counting function for the 
torus $\R^n/L$, where $L$ is a lattice, is equivalent to the lattice point counting problem in the 
ellipsoid associated with $L$. Thus here the error estimate in the lattice point counting problem 
amounts to estimating the error in Weyl's law for the corresponding torus.
\item When $B$ is given as the level set of a positive homogeneous form with integral coefficients 
and degree,  for example $x_1^k+\cdots + x_n^k$, the lattice point counting problem is equivalent 
to the fundamental number-theoretic problem of  bounding the error term in the average number of 
representations of positive integers by the form. 
\end{itemize}

Let us note that three of the motivating themes in the development of this subject have been :

\begin{enumerate}
\item Obtaining error estimates which are as sharp as possible in  the case of  Euclidean balls 
$tB^n\subset \R^n$. Here the best possible value of $\theta$ has been obtained for $n \ge 4$, and 
it is $\theta=2$ (namely $\kappa=\frac{n-2}{n}$). The conjectured value for $\R^2$ is $\theta=3/2$ 
(namely $\kappa=1/4$), and for $\R^3$ it is $\theta=2$ (namely $\kappa=1/3$). We refer to \cite{Kr1} 
and \cite{IKKN} for detailed information on the historical development and current best results.
\item Obtaining error estimates for Euclidean dilates of a general smooth compact convex body in 
$\R^n$ whose boundary has everywhere non-vanishing Gaussian curvature. Starting with \cite{Hl}, 
\cite{He}, many different results have been obtained, including, for example, for ellipsoids and 
other bodies of revolution. For more information we refer to \cite{IKKN} and the references therein, 
including \cite{Ch}.
\item Obtaining error estimates  for certain special bodies whose boundary surface contains points 
with vanishing Gaussian curvature. These include the unit balls of $\ell^p$-norms on $\R^n$ and some 
generalizations, and the effect of vanishing curvature on the error estimates have been investigated extensively 
in e.g. \cite{Kr2}, \cite{Kr3}, \cite{Kr4}, \cite{KN1}, \cite{KN2}, \cite{No}, \cite{Pe1}, \cite{Pe2}, 
\cite{R1}, \cite{R2}, \cite{R3} and \cite{R4}. 
\end{enumerate}

It is natural to consider the following considerably more general set-up. 
Let ${\sf G}\subset {\sf G}{\sf L}_n$ denote any connected linear algebraic group defined over $\Q$, 
such that the integral points $\Gamma={\sf G}(\Z)$ form  a lattice subgroup in the group 
$G={\sf G}(\R)\subset GL_n(\R)$ of real points. For interesting  gauge functions, for example a 
natural left-invariant distance $\text{dist}$ on $G$, a natural problem is to establish an asymptotic 
of the form 
\begin{equation}\label{asymptotic}
\abs{\Gamma\cap B_t} = m_G(B_t) + O\left(\left(m_G(B_t)\right)^\kappa\right)
\end{equation}
with $B_t=\set{g\in G \,:\, \text{dist}(g,e)< t}$, and $m_G$ Haar measure on $G$, normalized 
so that the measure of a fundamental domain of $\Gamma$ in $G$ has measure $1$. 

When the group in question is a (non-compact) semi-simple algebraic group, a general solution to the 
problem of estimating $\kappa$ has been developed in \cite{GN}. For simple group of real rank at least 
two, this estimate is the best currently available. However, it should be noted that the best possible 
$\kappa$ has never been established even in a single example, for any left-invariant distance on any 
(non-compact) simple Lie group.

Our purpose in the present paper is to investigate aspects of the lattice point counting problem 
(\ref{asymptotic}) on the Heisenberg groups. Let us first note that unlike the Euclidean case, this 
problem is completely different  from the eigenvalue counting problem for the natural Laplacian on 
compact Heisenberg homogeneous spaces $G/\Gamma$. The latter problem has been studied in detail, see 
\cite{KP} and the references therein. We note further that there has been considerable recent interest 
in geometric group theory  in specific lattice point counting results in the Heisenberg groups. 
These pertain to Carnot-Carath\'eodory distances arising from word metrics on the lattice subgroup, 
and we refer to \cite{BrD}, \cite{DM} and the references therein for more on this topic. These 
counting problems are completely different from those we will consider in the present paper. 

The lattice point counting problem on the Heisenberg groups that we will consider is that of 
counting in balls defined by natural radial Heisenberg-homogeneous gauge functions, and as far as 
the authors are aware, no prior results have been established regarding this problem. Let us now 
turn to describe our set-up, notation and results.

\subsection{The Heisenberg group} 
The Heisenberg group, denoted $\sf{H}_d,$ has several equivalent descriptions which we will use below. 
One is given by 
$$ {\sf H}_d=\R^d\times \R^d\times \R=\set{(x,y,t) \,:\, x,y\in \R^d, t\in \R}$$
with multiplication given by 
$$(x,y,t)(x',y',t')=(x+x',y+y', t+t'+\inn{x,y'}),$$ 
$\inn{x,y'}$ being the standard inner product on $\R^d$.\\
An equivalent formulation is given by the isomorphic group 
$$ {\sf H}_d=\C^d\times \R=\set{(z,t) \,:\, x+iy=z\in \C^d, t\in \R}$$ 
with multiplication 
$$(z,t)\cdot (z^\prime,t^\prime)=(z+z^\prime, \, t+t^\prime+2\Im z\cdot \bar{z}^\prime)$$
so that the multiplication can be also be described by the symplectic form :
$$(x,y,t)(x',y',t')=\left(x+x', \, y+y', \, t+t'+2(\inn{x,y'}-\inn{x',y})\right).$$ 

The \textbf{Heisenberg dilations} are defined by $(z,t)\mapsto \phi_a(z,t)=(az,a^2t)$ for any given 
$a\in \R_+$, and constitute a group of automorphisms of the Heisenberg group. Another important 
group of automorphisms of the Heisenberg group is the unitary group $U_d(\C)$, whose action is 
given by $(z,t) \mapsto (Uz,t)$, for $U \in U_d(\C)$. 

%
%
%
%
%
%

\subsubsection{Heisenberg-homogeneous radial norms}

The action of the dilation group gives rise to a natural notion of homogeneity on the Heisenberg group, 
where $f: {\sf H}_d\to \C$ is homogeneous of degree $\mu$ if it satisfies $f(az,a^2t)=a^\mu f(z,t)$. 
The action of $U_d(\C)$ gives rise to a natural notion of radiality on the Heisenberg group, where a function 
$f: {\sf H}_d\to \C$ is radial if it satisfies $f(Uz,t)=f(z,t)$, namely is invariant under 
the $U_d(\C)$-action. 

%
%
One of the most natural family of gauge functions on the Heisenberg group, the family we shall call 
Heisenberg norms, is given by, for $\alpha,A > 0$ :
$$N_{\alpha,A}((z,t))=\left(\abs{z}^\alpha + A \abs{t}^{\alpha/2}\right)^{1/\alpha}.$$
Clearly, Heisenberg norms $N_{\alpha,A}$ are radial and homogeneous of degree 1. For 
notational simplicity, at some places we will focus our attention on the family $N_\alpha = N_{\alpha,1}$, 
but the family $N_{\alpha,A}$ satisfies the same properties, and below we will point out briefly at 
the appropriate places that the arguments we use need only non-essential modifications. However, 
naturally the estimates involved are locally but not globally uniform in $A$. 

An interesting special case arises when $\alpha=2$, a gauge that was considered by a number of 
authors. For the balls associated with the Heisenberg norm $N_2$, we shall obtain the best possible 
result on the error estimate in the lattice point counting problem.

\subsubsection{The Cygan-Kor\'anyi Heisenberg-norm}
The most natural gauge function on the Heisenberg group arises when $\alpha=4$, namely 
$N_{4,A}((z,t)) = \left(\abs{z}^4 + A \abs{t}^2 \right)^{1/4}$.
This norm was considered by Cygan \cite{Cy}, \cite{Cy1} and Kor\'anyi \cite{Kor}, and is often referred to as 
the Kor\'anyi norm, or Cygan-Kor\'anyi norm (see \cite{PP}), which is the designation we shall adopt. 
To gauge its full significance the reader should consult \cite[\S 2, \S 3]{CDKR}, \cite{PP} and 
\cite{St}, and here let us just mention the following. View the Heisenberg group as part of the 
Iwasawa $AN$ group of the simple Lie group $SU(d,1)$, and embed it as a subset of the boundary of 
the complex hyperbolic space in the usual way. The Cygan-Kor\'anyi norm $N_{4,A}$, for suitable $A$ 
depending on the structure constant associated with the symplectic form defining the bracket in 
the Lie algebra, can then be characterized uniquely in geometric terms, and appears in the explicit 
expression defining the following canonical geometric objects :
\begin{enumerate}
\item the conformal inversion on $N\setminus\set{e}$, 
\item the Radon-Nikodym derivative of the conformal inversion acting on $N$,
\item the Busemann cocycle and the density of the Patterson-Sullivan measure on the boundary 
of complex hyperbolic space, 
\item the cross ratio on $N$. 
\end{enumerate}

In addition to the above, the Cygan-Kor\'anyi norm has attracted a lot of attention in the context 
of the harmonic analysis on the Heisenberg group. For example, it appears in the expression defining 
the fundamental solution of a natural sublaplacian on the Heisenberg group and in other natural kernels, 
see \cite{St} and \cite{Co} for an exposition and \cite{FL} and the references therein for recent results.

\subsection{Statement of the main result}
Let us recall that the Haar measure on ${\sf H}_d$ can be identified with the Lebesgue measure on 
$\R^{2d+1}$. We denote by $\abs{B}$ the Haar measure (= Euclidean volume) of a set $B\subset {\sf H}_d$, 
and recall that it scales under dilations according to the homogeneous dimension, not the Euclidean 
dimension. In particular, if 
$B_R^{\alpha,A} = \{(z,t) \in {\sf H}_d : \abs{z}^\alpha + A\abs{t}^{\alpha/2} \leq R^\alpha \}$ 
is the $N_{\alpha,A}$-ball of radius $R$ in ${\sf H}_d$, which is also the Heisenberg dilate of the unit 
ball namely $\phi_R(B_1^{\alpha,A})$, then $\abs{B_R^{\alpha,A}} = R^{2d+2} \abs{B_1^{\alpha,A}}$. 

\noindent {\bf Notation} : We let $\#(A)$ denote the size of a finite set $A\subset \mathbb{R}^{2d+1}$. 
We write $f \lesssim g$ if there exists a constant C such that $f \leq C g$, and $f \sim g$ if 
$f \lesssim g, \, g \lesssim f$.

Keeping the notation introduced above, we now state our main result on upper bounds on the error  
in the lattice point counting problem.
\begin{thm}\label{thm:main}
The error term in the lattice point counting problem in $B_R^{\alpha,A}$
is estimated as follows. 
\begin{enumerate}
\item For $d\geq 1$ and $\alpha = 2$, 
\begin{equation}\label{maind1} 
\abs{\# \left(\mathbb{Z}^{2d+1} \cap B_R^{2, A} \right) - \abs{B_R^{2,A}}} \lesssim R^{2d}\,.
\end{equation}
Furthermore, this result is the best possible.
\item For $d =1$ and $\alpha > 2$,
\begin{equation}\label{maind2} 
\abs{\# \left( \mathbb{Z}^{3} \cap B_R^{\alpha,A} \right) - \abs{B_R^{\alpha,A}}} \lesssim R^{2 + \max\{0, \, \delta(\alpha)\}} \log(R) \, .
\end{equation}
Here $\delta(\alpha)= \frac{1-\frac{4}{\alpha}}{\frac32-\frac{2}{\alpha}},$ so that in particular 
$\max\{0, \, \delta(\alpha)\} = 0$ for $2 < \alpha \le 4$.   
\item For  $d \geq 2$ and $\alpha > 0$, 
\begin{align}\label{maind3} 
\abs{\# \left(\mathbb{Z}^{2d+1} \cap B_R^{\alpha,A} \right) - \abs{B_R^{\alpha,A}}} \lesssim
\begin{cases}
R^{4} (\log(R))^{2/3} &; \mbox{ for } d = 2\\
R^{2d} &; \mbox{ for } d \geq 3 \, .
\end{cases}
\end{align}
\end{enumerate}
\end{thm}
We remark, that in the case of ${\sf H}_1$, when $\alpha$ is sufficiently large it is possible to 
improve the error estimate $R^{2+\delta(\alpha)}$ given above. We will explain this further in 
\S \ref{sec:proof-main}.  

\subsubsection{On the method of proof}

Let us first remark that since the Heisenberg group ${\sf H}_d$ is parametrized by Euclidean space 
$\R^{2d+1}$, and the lattice $\Gamma$ of integral points is parametrized by $\Z^{2d+1}$, the problem 
we consider can also be viewed as counting elements in the Euclidean lattice $\Z^{2d+1}$ contained 
in the family of increasing bodies $B_R^{\alpha,A} \subset R^{2d+1}$ as $R\to \infty$. 
Nevertheless, no Euclidean counting result of lattice points in dilates of convex bodies is directly 
relevant to our problem, since the Heisenberg dilations used to expand the given body $B_1^{\alpha,A}$ 
are materially different than the Euclidean dilations.

Our method of bounding the error term in the lattice point counting problem in 
$B_R^{\alpha,A} \subset {\sf H}_d$ uses a blend of Euclidean and Heisenberg notions. In \S 2, we will 
estimate the {\it Euclidean Fourier transform} of the characteristic function of $B_1^{\alpha,A}$. 
In \S 3, we will dominate the lattice point count in $B_R^{\alpha,A}$ from the above and below by the 
{\it Euclidean convolution} $\chi_{B_R^{\alpha,A}} \ast \rho_\epsilon$, where $\rho_\epsilon$ is a 
bump function. A key new point here is that $\rho_\epsilon$ is defined using  {\it Heisenberg dilations}, 
rather than the Euclidean ones. 
We will then apply the {\it Euclidean Poisson summation} formula to 
$\chi_{B_R^{\alpha,A}}\ast \rho_\epsilon$, and estimate the resulting product expression using the 
spectral decay estimates established in \S 2. We will  argue separately 
in several  different regions, whose structure reflects the fact that $\rho_\epsilon$ was defined 
using {\it Heisenberg dilations}. In \S 4, we will compare our upper bound on the error term with 
the lower bound that can be obtained from slicing - namely by viewing our lattice point problem 
in each hyperplane $(z,t)$ with $t$ fixed separately. In the hyperplane we apply the known results 
regarding the classical Euclidean sum-of-squares problem in $\R^{2d}$. 

Let us remark that estimating the decay of the Euclidean Fourier transform of $\chi_{B_1^{\alpha,A}}$ 
is crucial to our argument, but we have not found in the literature any result that applies directly 
to this problem. Indeed, as we shall see, it turns out that the bodies  $B_1^{\alpha,A}, \alpha \geq 2,$ are in fact 
Euclidean-convex bodies of revolution, but their surfaces have points of vanishing curvature, and 
this renders the elaborate results on the Fourier transform decay for bodies with surfaces of 
non-vanishing Gaussian curvature irrelevant. In fact, the surface of each $B_1^{\alpha,A}, \alpha > 2,$ has the 
property that the curvature vanishes to maximal order at some points, namely the Hessian is zero 
at those points. Our spectral decay estimates are direct and are based on the observation that the 
condition of radiality reduces the problem to estimating oscillatory integrals in the variables 
$(\abs{z},\abs{t})$. We analyze the latter using ideas developed in  estimating oscillatory integrals 
on plane curves initiated in \cite{SW}, using van-der-Corput classical results. The estimates we 
obtain exceed, in our particular situation, those that can be deduced from the current standard 
estimates for the decay of the Fourier transform of a general Euclidean convex body whose surface 
has points of zero Gaussian curvature to maximal order. 

The closest point of comparison to our spectral decay result for $B_R^{\alpha,A}$ would seem to be 
spectral decay results for special Euclidean bodies such as $\ell^p$-balls and other related bodies, 
which were considered in e.g. \cite{R1}, \cite{R2}, \cite{Kr2}, \cite{Kr3}, \cite{KN1}, \cite{KN2}. 
As we shall show in \S 6, our method, when applied to the Euclidean lattice point counting problem 
in some of these bodies, actually yields the same (main) error estimate obtained for some of them 
in the references cited.  We should note however that the analysis in these references is much more 
elaborate and produces a secondary summand in the asymptotic expansion.

\section{Harmonic analysis of Heisenberg norm balls} \label{Heisenberg-spectral}

\subsection{Properties of Heisenberg norm balls}
Let us begin by establishing three properties of the norm balls $B_1^{\alpha,A} = 
\{(z,t) \in {\sf H}_d : N_{\alpha,A}((z,t)) \leq 1 \}$ that are the subject of our discussion. 
Essential use will be made of the third property later on. 

\subsubsection{Euclidean convexity of the $N_{\alpha,A}$-norm balls}

\begin{prop} \label{convexity}
In every Heisenberg group ${\sf H}_d$, $d \ge 1$, the unit balls $B_1^{\alpha,A}$ are 
Euclidean-convex if and only if $\alpha \geq 2$. 
\end{prop}
\begin{proof}
Let $\alpha \ge 2$,  fix $(z,t), (w,s) \in B_1^{\alpha,A}$ and $0 < \lambda < 1$, and write 
\begin{align*}
\left(N_{\alpha,A}(\lambda(z,t) + (1-\lambda)(w,s))\right)^{\alpha} &= \left(N_{\alpha,A}((\lambda z + (1-\lambda) w, \lambda t + (1-\lambda)s))\right)^{\alpha}\\
&= \abs{\lambda z + (1-\lambda)w}^\alpha + A\abs{\lambda t + (1-\lambda)s}^{\alpha/2}\\
&\leq \left(\lambda \abs{z} + (1-\lambda) \abs{w} \right)^\alpha + A\left(\lambda \abs{t} + 
(1-\lambda)\abs{s} \right)^{\alpha/2}.
\end{align*}

Using the convexity of $x\mapsto x^\alpha $ and $x\mapsto x^{\alpha/2}$ for $\alpha \ge 2$, the 
latter expression is bounded by 
$$\lambda \abs{z}^\alpha + (1-\lambda) \abs{w}^\alpha +
A\lambda \abs{t}^{\alpha/2} + A(1-\lambda) \abs{s}^{\alpha/2}$$
which equals 
$$\lambda \left(N_{\alpha,A}((z,t))\right)^\alpha + (1-\lambda) \left(N_{\alpha,A}((w,s))\right)^\alpha .$$
Thus if $N_{\alpha,A}((z,t)) \le 1$ and $N_{\alpha,A}((w,s)) \le 1$ then also $N_{\alpha,A}(\lambda(z,t) + (1-\lambda)(w,s))\le 1$, 
and the unit $N_{\alpha,A}$-norm ball is Euclidean-convex.

To see the non-convexity of $B_1^{\alpha,A}$ in the
case $\alpha <2$, consider the set $D_{\alpha,A} = \{(a,b) \in \R^2 : a,b \geq 0 \textup{ and } a^\alpha
+ A \, b^{\alpha/2} \le1 \}$. It can be isometrically  identified with the intersection
of the unit ball $B_1^{\alpha,A}$ with the set $U = \{(z,t) : z=a(1,0,...,0) \textup{ and }a,t \geq 0\}$, 
and $U$ is clearly a Euclidean-convex set. If the unit ball $B_1^{\alpha,A}$ was Euclidean-convex as well, 
then $D_{\alpha,A}$ being the intersection of two Euclidean-convex sets would also be a Euclidean-convex 
subset of $\RR^2$. But, for $\alpha < 2$ this is not the case as can easily be verified directly. 
Thus the unit balls for the $N_{\alpha,A}$-norm are convex if and only if $\alpha \geq 2$.
\end{proof}

\subsubsection{Vanishing of principal and Gaussian curvatures for $N_{\alpha,A}$-norm balls}

Considering the curvature of the surface bounding the body $B_1^{\alpha,A}$ for $d\ge 1$ and $\alpha >2$, 
we note the following. 

{\it The north and south poles}. The Gaussian curvature of the surface of $B_1^{\alpha,A} \subset \R^{2d+1}$ vanish at 
both the points of intersection of $B_1^{\alpha,A}$ and the $t$-axis, namely 
the north and south poles. In fact all of the $2d$ principal curvatures vanish at these two points, 
so that the Hessian of the defining equation at these points is the zero matrix. In view of the symmetry of 
the surface, it is enough to compute the principal curvatures at the point $t=-1.$ 
The surface near the point $t = -1$ (after translation) is given by
$$t= \varphi(X) = \varphi(X_1, \ldots, X_{2d}) = A^{-2/\alpha} \left(1-\left(1- \abs{X}^\alpha \right)^{2/\alpha}\right),$$
with $\varphi(\vec{0}) = 0 = \grad\varphi(\vec{0})$. Differentiating directly, the Hessian matrix
$H = \left(\frac{\partial^2\varphi}{\partial X_i \partial X_j}\right)$
obtained at the origin is the zero matrix, so that the principal curvatures at the origin all vanish, 
being the eigenvalues of the Hessian. 
%
%

{\it The equator.} On the equator of the surface, namely the intersection of $B_1^{\alpha,A}$ with the 
hyperplane $t=0$, the Gaussian curvature vanishes as well, but here only one principal curvature 
vanishes. Indeed, isolating the first variable $X_1$,  the surface near the point $X_1 = 1$ 
(after translation) is given by
$$X_1 = \psi(X_2, \ldots, X_{2d}, t) = 1-\left(\left(1-A\abs{t}^{\alpha/2}\right)^{2/\alpha} - 
\left(X_2^2+ \cdots +X_{2d}^2\right)\right)^{1/2},$$
with $\psi(\vec{0}) = 0 = \grad\psi(\vec{0})$. Differentiating directly, the Hessian matrix at origin is in fact diagonal, 
and only one diagonal entry is 0, namely $ \psi_{tt}$. All other diagonal entries, i.e. $\psi_{X_j X_j}$ 
are non-zero. Thus, the equator forms a curve where the Gaussian curvature vanishes to the first order.

\subsubsection{Euclidean subadditivity of the $N_{\alpha,A}$-norm balls}

Our analysis below will involve Euclidean convolution with a special family $\rho_\epsilon$ of bump 
functions, defined as follows.  Let us fix a bump function 
$\rho: \mathbb{R}^{2d}\times \mathbb{R}\rightarrow \mathbb{R}$, which is a smooth non-negative 
function with support contained in the unit ball $B_1^{\alpha,A}$, such that $\rho(0) > 0$ and 
$\int_{B_1^{\alpha,A}} \rho(z,t) \,dz \, dt =1$. We then consider the family of functions defined by the 
normalized Heisenberg dilates of $\rho$, namely 
$$\rho_{\epsilon}(z,t) = \frac{1}{\epsilon^{2d+2}} \rho\left(\frac{z}{\epsilon}, \frac{t}{\epsilon^2}\right) 
= \frac{1}{\epsilon^{2d+2}} \rho \circ \phi_{1/\epsilon}(z,t).$$
Clearly, $\rho_\epsilon$ is supported in the ball $B^{\alpha,A}_\epsilon$, and of course, 
$\int_{\R^{2d+1}} \rho_\epsilon (z,t) \, dz \, dt =1$ for all $\epsilon > 0$. 

We can now state the following :
\begin{prop}\label{triangle} For every $d \ge 1$,  and $\alpha \ge 1$, 
\begin{enumerate}
\item the $N_{\alpha,A}$-norms are subadditive on $\R^{2d+1}$ with respect to Euclidean addition, namely :
$$N_{\alpha,A}((z,t) + (w,s)) \le N_{\alpha,A}((z,t)) + N_{\alpha,A}((w,s)).$$
\item The balls $B_R^{\alpha,A}$ satisfy the following two-sided inequalitiy with respect to Euclidean convolution : 
$\chi_{B_{R - \epsilon}^{\alpha,A}} * \rho_\epsilon \leq \chi_{B_R^{\alpha,A}} \leq \chi_{B_{R + \epsilon}^{\alpha,A}} * \rho_\epsilon.$
\end{enumerate}
\end{prop}
\begin{proof}
To see that the Heisenberg norms are subadditive with respect to Euclidean addition, fix any 
$(z,t), (w,s) \in \R^{2d+1}={\sf H}_d$. Then
\begin{eqnarray*}
N_{\alpha,A}((z,t) + (w,s)) &=& \left(\abs{z+w}^\alpha + A\abs{t+s}^{\alpha/2}\right)^{1/\alpha}\\
&\leq& \left( (\abs{z} + \abs{w})^\alpha + A\left((\abs{t}+\abs{s})^{1/2}\right)^\alpha \right)^{1/\alpha}\\
&\leq& \left((\abs{z} + \abs{w})^\alpha + \left(A^{1/\alpha}\abs{t}^{1/2} + A^{1/\alpha}\abs{s}^{1/2}\right)^\alpha\right)^{1/\alpha}\\
&\leq& \left(\abs{z}^\alpha + A\abs{t}^{\alpha/2}\right)^{1/\alpha} + \left(\abs{w}^\alpha + A\abs{s}^{\alpha/2}\right)^{1/\alpha}\\
&=& N_{\alpha,A}((z,t)) + N_{\alpha,A}((w,s)) \, ,
\end{eqnarray*}
where we have used $(a+b)^{1/2}\le a^{1/2}+b^{1/2}$ for $a, b \ge 0$, and following that we have 
used the triangle inequality in $l_\alpha(\mathbb{R}^2)$ applied to the vectors $\left(\abs{z}, A^{1/\alpha} \abs{t}^{1/2}\right)$ 
and $\left(\abs{w}, A^{1/\alpha} \abs{s}^{1/2}\right).$

To prove the second inequality of part (2), note first that if $(z,t) \notin B^{\alpha,A}_R$, 
then the inequality holds trivially. Taking $(z,t) \in B^{\alpha,A}_R$, we have 
\begin{eqnarray*}
\chi_{B_{R + \epsilon}^{\alpha,A}} \ast \rho_\epsilon (z,t) &=& \int_{B_{R + \epsilon}^{\alpha,A}} \rho_\epsilon (z-w, t-s)\, dw \, ds \\
&=& \int_{(z,t)-B_{R+\epsilon}^{\alpha,A}} \rho_\epsilon(w,s) \, dw \, ds \geq \int_{B_\epsilon^{\alpha,A}} \rho_\epsilon (w,s) \, dw \, ds \, ,
\end{eqnarray*}
where the last step holds true because of the non-negativity of $\rho_\epsilon$ and the Euclidean 
subadditivity of the $N_{\alpha,A}$-norms, which implies that for $(z,t)\in B_R^{\alpha,A}$, the set 
$(z,t)-B_{R + \epsilon}^{\alpha,A}$ contains $B_\epsilon^{\alpha,A}$. Thus,
\begin{eqnarray*}
\chi_{B_{R + \epsilon}^{\alpha,A}} \ast \rho_\epsilon (z,t) \geq \int_{B_\epsilon^{\alpha,A}} \rho_\epsilon (w,s) \, dw \, ds
= \int_{B_1^{\alpha,A}} \rho (w,s) \, dw \, ds = 1 = \chi_{B_R^{\alpha,A}}(z,t).
\end{eqnarray*}
Finally, note that for any $(z,t) \in B^{\alpha,A}_R$, the first inequality follows from the fact that 
$\chi_{B_{R - \epsilon}^{\alpha,A}} \ast \rho_\epsilon \leq \int \rho_\epsilon = 1$, whereas for 
$(z,t) \notin B^{\alpha,A}_R$, 
\begin{eqnarray*}
\chi_{B_{R - \epsilon}^{\alpha,A}} \ast \rho_\epsilon (z,t) &=& \int_{(z,t)-B_{R-\epsilon}^{\alpha,A}} \rho_\epsilon(w,s) \, dw \, ds 
\leq \int_{\R^{2d+1} \setminus B_\epsilon^{\alpha,A}} \rho_\epsilon (w,s) \, dw \, ds = 0.
\end{eqnarray*}
This completes the proof of Proposition \ref{triangle}.
\end{proof}


\subsection{Decay of Fourier transforms of Heisenberg norm balls}
We continue to view the Heisenberg-homogeneous norm balls
$$B_1^{\alpha,A}=\{(w,s) \in \mathbb{R}^{2d}\times \mathbb{R} : \abs{w}^{\alpha} + A \abs{s}^{\alpha/2} \le 1\}\,,$$
as subsets of $\R^{2d+1}$, and in the present section we will give estimates on the rate of decay of 
the Euclidean Fourier transform of $\chi_{B_1^{\alpha,A}}$. These decay estimates will be applied in the 
next section to the problem of finding the error term in the lattice point counting problem in the 
sets $B_R^{\alpha,A} = \phi_R(B_1^{\alpha,A})$, where $\phi_R(z,t)=(Rz,R^2 t) $ is the Heisenberg dilation.

The unitary characters of $\R^{2d+1}$ are parametrized by $\set{(w,s)\,:\, w\in \R^{2d}, s\in \R}$.
Let $\widehat{f}$ denote the Euclidean Fourier transform of a function $f$, in the form
$$\widehat{f}(w,s)=\int_{\R^{2d+1}}e^{-2\pi i (\inn{z,w}+ts)}f(z,t)\, dz \, dt \,.$$
We first record here the following useful identity :
\begin{equation} \label{FT-A-1}
\widehat{\chi_{B_1^{\alpha,A}}} (w,s) = A^{-2/\alpha} \, \widehat{\chi_{B_1^\alpha}} \left(w, A^{-2/\alpha}s\right)
\end{equation}
which relates the Euclidean Fourier transform of $\chi_{B_1^{\alpha,A}}$ with that of $\chi_{B_1^{\alpha}}$. 
We divide the parameter space to three subsets and will argue separately in each. The domains of our
consideration are given by 
\begin{enumerate}
\item  $w= \vec{0}$, namely the $s$-axis, 
\item  the hyperplane $s=0$, 
\item  the set $\abs{w}\ge 1$ and  $\abs{s}\ge A^{2/\alpha}$.
\end{enumerate}

As we shall see below, this decomposition is natural in the context of the Heisenberg group and is 
dictated by the decomposition of $\R^{2d+1}$ to eigenspaces of the Heisenberg dilation. Furthermore, 
since after applying the Euclidean Poisson summation formula we will be interested in summing the 
Fourier transform over the points in the integer (dual) lattice, estimates on the sets listed above 
will suffice.

\noindent {\bf Notation} : We write $B_1^{\alpha}$ and $B_R^{\alpha}$ to denote $B_1^{\alpha,1}$ and $B_R^{\alpha,1}$ respectively.

\subsubsection{Decay of the Fourier transform along $s$-axis} 

\begin{lem} \label{taxis}
For $\alpha > 0, \abs{s} \geq A^{2/\alpha},$
$$\abs{\widehat{\chi_{B_1^{\alpha,A}}}(\vec{0},s)} \lesssim \abs{s}^{-\left(1+\min \left\{\frac{2d}{\alpha}
, \, \frac{\alpha}{2}\right\}\right)} \, .$$
However, for a positive integer $\alpha \in 4 \N$ we have a better estimate
$$\abs{\widehat{\chi_{B_1^{\alpha,A}}}(\vec{0},s)} \lesssim \abs{s}^{-\left(1+\frac{2d}{\alpha}\right)} \, .$$
\end{lem}
\begin{proof}
In view of identity \eqref{FT-A-1}, it suffices to establish the lemma for $A=1$. Now,
\begin{eqnarray*}
\widehat{\chi_{B_1^\alpha}}(\vec{0},s) &=& \int_{B_1^\alpha} e^{-2\pi its} \, dz \, dt \\
&=& \int_{-1}^1 e^{-2\pi i t s} \left( \int_{\abs{z} \leq \left(1-\abs{t}^{\alpha/2} \right)^{1/\alpha}}\, dz \right)\, dt \\
&=& C_d \int_{-1}^1 e^{-2\pi i t s} \left(1-\abs{t}^{\alpha/2} \right)^{2d/\alpha} \, dt \, .
\end{eqnarray*}
Before beginning the proof, let us first note that for $\alpha = 4$, the integral above is, up to a 
constant, the Fourier transform of the Euclidean unit ball in $\mathbb{R}^{d+1}$ evaluated at 
$(\vec{0},s)$; in this case, the decay is well-known and agrees with the claim.
Proceeding with the analysis of a general $\alpha \ge 2$ and $B_1^\alpha$, we consider the following 
two cases :

\vspace{1em} 

\underline{Case I:} When $\alpha = 4k$ for some positive integer $k$, then we write 
$\frac{2d}{\alpha} = \frac{d}{2k} = m+\mu$ for $m \in \N, \, 0<\mu\leq 1$ and apply 
integration by parts $m+1$-times to get
\begin{eqnarray*}
\widehat{\chi_{B_1^\alpha}}(\vec{0},s) &=& \frac{C_{k, d}}{s^{m+1}} \int_{-1}^1 e^{-2\pi i t s}
\left(1-t^{2k} \right)^{\mu-1} P(t) \, dt \\
&=& \frac{C_{k, d}}{s^{m+1}} \int_0^1 e^{-2\pi i t s} (1-t)^{\mu-1}
\left\{\left( \sum_{j=0}^{2k-1} t^j \right)^{\mu-1} P(t)\right\} \, dt \\
&&+ \frac{C_{k, d}}{s^{m+1}} \int_{-1}^0 e^{-2\pi i t s} (1+t)^{\mu-1}
\left\{\left( \sum_{j=0}^{2k-1} (-t)^j \right)^{\mu-1} P(t)\right\} \, dt
\end{eqnarray*}
for some polynomial $P$. 
We now use the standard estimates of oscillatory integrals stated below, which apply to integrals having singularity at one of the end 
points. In case $\mu = 1$, the integrals have no singularity and we integrate by parts once more.
%
Thus, we have shown that the above integrals decay at least at the rate of $\abs{s}^{-\mu}$ 
which proves the claimed estimate in the case of an integer $\alpha \in 4 \N$.

\vspace{1em} 

\underline{Case II:} In general, we consider the following integral for $\beta, \gamma > 0$ : 
\begin{eqnarray*}
I_{\beta, \gamma}(s) = \int_0^1 \cos(ts) \left(1-t^\beta \right)^\gamma \, dt \, .
\end{eqnarray*}
Let $q \geq 1$ be the integer for which $-1 < \min\{\beta-q, \gamma-q\} \leq 0 \, .$ 
Applying integration by parts $q$-times we get
\begin{align*}
I_{\beta, \gamma}(s) =& \frac{1}{s^q} \int_0^1 F(ts)t^{\beta -q} \left(1-t^\beta \right)^{\gamma-q}
\left\{\sum_{j=0}^{q-1} C_{q,j} (t^\beta)^j \right\} \, dt \\
=& \frac{1}{s^q} \int_{\frac{1}{2}}^1 F(ts) \left(1-t\right)^{\gamma-q} \left\{\left(\frac{1-t^\beta}{1-t}\right)^{\gamma-q}
\sum_{j=1}^{q} C_{q,j} t^{j\beta -q} \right\} \, dt\\
&+ \frac{1}{s^q} \sum_{j=1}^{q} C_{q,j} \int_0^{\frac{1}{2}} F(ts) t^{j\beta -q}
\left\{ \left(1-t^\beta \right)^{\gamma-q} \right\} \, dt.
\end{align*}
Here $F(s) = \cos(s)$ if $q$ is even, otherwise $F(s) = \sin(s)$. 
%
%
An argument similar to that of Case I then implies that the above integrals in the expression of $I_{\beta, \gamma}(s)$ 
decay at least at the rate of $\abs{s}^{-(1+\min\{\beta-q, \, \gamma-q\})}$. Substituting the appropriate 
values of $\beta$ and $\gamma$ completes the proof of Lemma \ref{taxis}.
\end{proof}

\subsubsection{Estimates of oscillatory integrals with singularities at one endpoint.} \label{vdC}
It is well known that one can use the method of staionary phase to prove the following estimates 
(for full  details regarding even more refined asymptotic estimates, see e.g. \cite{Er}, Sec 2.8).
For any $0 < \lambda \le 1$, and $-\infty < a <b < \infty$, we have
\begin{equation} \label{oscillatory-1}
\abs{\int_a^b e^{-i s t} (t-a)^{\lambda-1}\, dt} \leq C_{\lambda} \abs{s}^{-\lambda}. 
\end{equation}
Here the constant $C_{\lambda}$ does not depend on $a$ and $b$.
From this one can easily deduce that for any differentiable function $g$ on $[a,b]$ such that 
$g^\prime $ is integrable on $[a,b]$
\begin{equation} \label{oscillatory-2}
\abs{\int_a^b e^{-i s t} (t-a)^{\lambda-1} g(t) \, dt} \leq
C_{\lambda} \left(\abs{g(b)} + \int_a^b \abs{g^\prime(t)} \, dt \right) \abs{s}^{-\lambda}.
\end{equation}
In fact, the integral in \eqref{oscillatory-2} can be written as $\int_a^b F^\prime(t) g(t) \, dt$ with
$$F(t) = \int_a^t e^{-i s r} (r-a)^{\lambda-1} \, dr \, ; \, \hspace{0.2in} a<t<b.$$
The estimate of \eqref{oscillatory-2} then follows by doing integration by parts and using 
\eqref{oscillatory-1}.

\subsubsection{Decay of the Fourier transform on the hyperplane $s$=0} \label{s=0} 
We begin by re-writing $\widehat{\chi_{B_1^\alpha}}(w,s)$ as follows, and we use this expression 
in the next two lemmas. 
\begin{eqnarray*}
\widehat{\chi_{B_1^\alpha}}(w,s) &=& \int_{B_1^\alpha} e^{-2\pi i (\inn{z,w}+ts)} \, dz \, dt \\
&=& \int_{\abs{z} \leq 1} 
e^{-2\pi i \inn{z,w}} \left( \int_{\abs{t} \leq \left(1-\abs{z}^{\alpha} \right)^{2/\alpha}} e^{-2\pi i ts} \, dt \right) \, dz \\
&=& \int_{0}^{1}  \left( \int_{S^{2d-1}} e^{-2\pi i r\theta \cdot w}  d\theta \right) \left( \int_{\abs{t} 
\leq \left(1-r^\alpha\right)^{2/\alpha}} e^{-2\pi i t s} \, dt \right) r^{2d-1} \, dr \\
&=& \int_{0}^{1}  \left(\int_{S^{2d-1}} e^{-2\pi i r\theta \cdot w}  d\theta \right) 
\frac{2 \sin\left(2\pi s(1-r^\alpha)^{2/\alpha}\right)}{2\pi s} r^{2d-1} \, dr \\
&=& c_d \int_{0}^{1}  \widehat{\sigma}(2\pi rw) \frac{2 \sin\left(2\pi s(1-r^\alpha)^{2/\alpha}\right)}{2\pi s} r^{2d-1} \, dr \, ,
\end{eqnarray*}
where $\sigma$ denotes the surface measure on the $(2d-1)-$dimensional sphere. \\
Notice that by the rotation invariance of $\sigma$, and hence of $\widehat{\sigma}$, 
\begin{align*}
\widehat{\sigma}(rw) &= \int_{S^{2d-1}} e^{-i r \abs{w} \theta \cdot e_1} \, d\theta\\
&=c_d \int_{0}^{\pi} e^{-i r \abs{w} cos(\phi)} (sin(\phi))^{2d-2} \, d\phi\\
&= c_d\int_{-1}^{1} e^{-i r \abs{w} u} (1-u^2)^{\frac{2d-3}{2}} du.
\end{align*}
By definition of the Bessel function $J_{d-1}$, the last integral equals to a constant multiple 
of $(2\pi r \abs{w})^{-(d-1)} J_{d-1}(r \abs{w})$.

We now state the estimate along the hyperplane.

\begin{lem} \label{xyaxis}
For any $\alpha > 2$ we have for $\abs{w} \geq 1$,
\begin{align*}
\abs{\widehat{\chi_{B_1^{\alpha,A}}}(w,0)} \lesssim
\begin{cases}
\abs{w}^{-2d} &; \mbox{ if } \frac{2}{\alpha} > d-\frac{1}{2} \\
\abs{w}^{-\left(d+\frac{1}{2}+\frac{2}{\alpha}\right)} &; \mbox{ if } \frac{2}{\alpha} \leq d-\frac{1}{2} \, .
\end{cases}
\end{align*}
Moreover, for $\alpha =2$, one has the following identity
\begin{eqnarray*}
\widehat{\chi_{B_1^{2,A}}}(w,0) = C_{d,A} \, \frac{J_{d+1}(2\pi \abs{w})}{\abs{w}^{d+1}} \, .
\end{eqnarray*}
\end{lem}

\begin{proof}
In view of \eqref{FT-A-1}, it again suffices to prove the lemma for $A=1$. Putting $s=0$ in the expression derived for $\widehat{\chi_{B_1^\alpha}}(w,s)$ in the beginning of this 
section, we have
$$ \widehat{\chi_{B_1^\alpha}}(w,0) =  \frac{C_d}{\abs{w}^{d-1}} \int_0^1 J_{d-1}(r \abs{w}) \left(1-r^\alpha \right)^{2/\alpha} r^d \, dr.$$
Recall first that the Bessel functions satisfy following identity (see e.g \cite{Gr} Appendix B.3, p. 427) :
\begin{eqnarray*}
\int_0^1 J_\mu(rs) (1-r^2)^{\nu} r^{\mu+1} \, dr = C_{\nu} \frac{J_{\mu+\nu+1}(s)}{s^{\nu+1}} \, ,
\end{eqnarray*}
for any $\mu > -1/2, \, \nu > -1.$ From this identity we get in the case $\alpha =2,$
\begin{eqnarray*}
\widehat{\chi_{B_1^2}}(w,0) = C_d \frac{J_{d+1}(2\pi \abs{w})}{\abs{w}^{d+1}} \, .
\end{eqnarray*}
In general, for any $\alpha >2$, we are left to study the bound (as $1 \leq \xi \to \infty$) for
$$I(\xi) = \int_0^1 J_{d-1}(\xi r) \left(1-r^\alpha \right)^{2/\alpha} r^d \, dr.$$
The proof will be completed once we show that
\begin{eqnarray*}
\abs{I(\xi)} \lesssim
\begin{cases}
\xi^{-(d+1)} &; \mbox{ if } \frac{2}{\alpha} > d-\frac{1}{2} \\
\xi^{-(\frac{3}{2} + \frac{2}{\alpha})} &; \mbox{ if } \frac{2}{\alpha} \leq d-\frac{1}{2}.
\end{cases}
\end{eqnarray*}
In order to prove the above estimate, we utilize the asymptotic expansion of $J_{d-1}$ which is 
valid when $\xi r\ge 1$.  Therefore, we divide the integration over $ r\in [0, 1]$ to two parts : 
on  the interval $[0,\delta]$ and  on the interval $[\delta, 1]$, where $\delta \geq \xi^{-1}$ is 
fixed and will be chosen momentarily. Since $J_{d-1}$ is a bounded function, we clearly have
\begin{eqnarray*}
\abs{\int_0^{\delta} J_{d-1}(\xi r) \left(1-r^\alpha \right)^{2/\alpha} r^d \, dr} \lesssim \int_0^\delta r^d dr \lesssim \, \delta^{d+1}\,.
\end{eqnarray*}
To estimate  the second interval, we recall the following asymptotic expression for 
$J_{d-1} (\xi r)$, valid when $\xi r \ge 1$   \cite[p. 356]{St} :
\begin{eqnarray*}
J_{d-1}(\xi r) &\sim& \left(\frac{\pi \xi r}{2}\right)^{-1/2} \cos \left(\xi r - \frac{(d-1)\pi}{2}
- \frac{\pi}{4}\right) \sum_{j=0}^\infty a_j \left(\xi r\right)^{-2j} \\
&& + \left(\frac{\pi \xi r}{2}\right)^{-1/2} \sin \left(\xi r - \frac{(d-1)\pi}{2} - \frac{\pi}{4}\right)
\sum_{j=0}^\infty b_j \left(\xi r\right)^{-2j-1}.
\end{eqnarray*}
with the implied constant independent of $\xi r$. 

We make use of the first two terms in the above aysmptotic expression. The estimate in question 
then becomes equivalent to
$$\int_{\delta}^1 \left[e^{i \xi r} \left(    \sum_{k=0}^{1} c_j (\xi r)^{-\left(k+ \frac{1}{2}\right)}\right) +
O\left((\xi r)^{-5/2}\right)\right] (1-r^\alpha)^{2/\alpha} r^d \, dr \,.$$
We estimate the $O$-term as follows :
\begin{eqnarray*}
\abs{\xi^{-5/2} \int_{\delta}^1 r^{d-5/2} \, dr} = \xi^{-5/2}
\abs{\frac{1-\delta^{d-3/2}}{d-3/2}} \leq 2
\begin{cases}
\xi^{-5/2} \delta^{-1/2} &; \mbox{ if } d=1 \\
\xi^{-5/2} &; \mbox{ if } d>1 .\\
\end{cases}
\end{eqnarray*}
%
The main term can be written as the difference of two integrals over the intervals of $[0,1]$
and $[0,\delta]$. The integral over the interval $[0,\delta]$ is bounded by
\begin{eqnarray*}
\sum_{k=0}^1 \abs{c_j} \xi^{-\left(k+ \frac{1}{2}\right)} \int_0^\delta r^{d-k-\frac{1}{2}} \, dr
= \sum_{k=0}^1 \frac{\abs{c_j}}{d-k+\frac{1}{2}} \xi^{-\left(k+ \frac{1}{2}\right)} \delta^{d-k+\frac{1}{2}} \, .
\end{eqnarray*}
Finally we proceed with the main term over the interval $[0,1]$ :
$$ \sum_{k=0}^1 c_j \xi^{-\left(k+ \frac{1}{2}\right)} \int_0^1 e^{i \xi r}
r^{d-k-\frac{1}{2}} (1-r^\alpha)^{2/\alpha} \, dr \,.$$
Now an argument similar to that of Case II in the proof of Lemma \ref{taxis} can be given to prove that
$$\abs{\int_0^1 e^{i \xi r} r^{d-k-\frac{1}{2}} (1-r^\alpha)^{2/\alpha} \, dr} \lesssim 
\xi^{-\left(1+ \min \left\{d-k-\frac{1}{2}, \, \frac{2}{\alpha}\right\}\right)} \, .$$
Collecting all the bounds, we see that $I(\xi) = \int_0^1 J_{d-1}(\xi r) \left(1-r^\alpha
\right)^{2/\alpha} r^d dr$ is dominated by a finite sum of terms of the form
$$\delta^{d+1}, \, \xi^{-\frac{5}{2}} \delta^{-\frac{1}{2}}, \, \xi^{-\left(k+\frac{1}{2}\right)} \delta^{d-k+\frac{1}{2}}, \, 
\xi^{-\left(k+\frac{3}{2}\right)} \xi^{-\min\left\{ d-k-\frac{1}{2}, \, \frac{2}{\alpha} \right\} }$$
for $k = 0,1.$ Choosing $\delta = \xi^{-1}$, we verify that
\begin{eqnarray*}
\abs{I(\xi)} \lesssim
\begin{cases}
\xi^{-(d+1)} &; \mbox{ if } \frac{2}{\alpha} > d-\frac{1}{2} \\
\xi^{-(\frac{3}{2} + \frac{2}{\alpha})} &; \mbox{ if } \frac{2}{\alpha} \leq d-\frac{1}{2}.
\end{cases}
\end{eqnarray*}
This completes the of proof Lemma \ref{xyaxis}.
\end{proof}


\subsubsection{Completion of the Fourier transform decay estimates}
The third case we consider is establishing decay when $\abs{w} \ge 1$ and $\abs{s}\ge A^{2/\alpha}	$. 
\begin{lem} \label{nonzero}
For any $\alpha \geq 2$ and  $\abs{w} \geq 1$ and $\abs{s} \geq A^{2/\alpha}$,
\begin{eqnarray}
\abs{\widehat{\chi_{B_1^{\alpha,A}}}(w,s)} \lesssim  \abs{w}^{-d} \abs{s}^{-1}.
\end{eqnarray}
Moreover, for $\alpha =2,$ we have the better decay estimate
\begin{eqnarray}
\abs{\widehat{\chi_{B_1^{2,A}}}(w,s)} \lesssim \abs{w}^{-\left(d-\frac{1}{2}\right)} \abs{s}^{-1} \abs{(w,s)}^{-1/2}.
\end{eqnarray}
\end{lem}


\begin{proof}
As observed in the earlier lemmas, here also it suffices to prove the estimates for $A=1$, thanks 
to \eqref{FT-A-1}. As mentioned in the beginning of \S \ref{s=0}, we have
$$\widehat{\chi_{B_1^\alpha}}(w,s) = c_d \abs{w}^{-(d-1)} s^{-1} \int_0^1 J_{d-1}(2\pi r \abs{w}) 
\sin\left(2\pi s(1-r^\alpha)^{2/\alpha}\right) r^d \, dr\,.$$ 

We divide the integration over $ r\in [0, 1]$ to two parts : on  the interval $[0,\delta]$ and on 
the interval $[\delta, 1]$. $J_{d-1}(\cdot)$ and $\sin (\cdot)$ being  bounded functions on the 
interval, we clearly have 
$$\abs{\int_0^{\delta} J_{d-1}(2\pi r \abs{w}) \sin\left(2\pi s(1-r^\alpha)^{2/\alpha}\right) r^d \, dr} 
\lesssim \int_0^{\delta} r^d \, dr = \frac{\delta^{d+1}}{d+1}\,.$$

To estimate the integral on the interval $[\delta, 1]$, we once again make use of the asymptotic
expression for the Bessel function $J_{d-1} (\cdot)$ as we did in the proof of Lemma \ref{xyaxis}.
This time we make use only of the first term. 
This integral can be written as the difference of two integrals over the intervals of $[0,1]$
and $[0,\delta]$. The integral over the interval $[0,\delta]$ is bounded by
$$\abs{w}^{-1/2} \delta^{d+1/2}.$$
The integral over $[\delta,1]$ is then bounded by
$$\left(\abs{w}^{-1/2} \int_0^1 e^{2\pi i \abs{w} r} \sin\left(2\pi s(1-r^\alpha)^{2/\alpha}\right) r^{d-\frac{1}{2}} \, dr\right) 
+ \abs{w}^{-1/2}\delta^{d+1/2}.$$
Now it suffices to estimate (from here onwards we drop the constant $2\pi$ for the sake of convenience)
$$\abs{w}^{-1/2} \int_0^1 e^{i \abs{w} r}e^{is (1-r^\alpha)^{2/\alpha}} r^{d-\frac12} \, dr \, ,$$
in both cases when $s$ is positive or negative. The above integral is the same as
\begin{eqnarray} \label{nonzeroeq-oscillatory}
\abs{w}^{-1/2} \int_0^1 \exp\left(i \abs{w} \phi_{\abs{w}, s}(r)\right) r^{d-\frac12} \, dr \, ,
\end{eqnarray}
where the phase function is given by $\phi_{\abs{w}, s}(r) = r + \frac{s}{\abs{w}} (1-r^\alpha)^{2/\alpha}.$ 

We now use the van-der-Corput lemma to estimate this integral. Note that the derivative of the phase 
function is given by 
$$ \phi_{\abs{w}, s}^{\prime}(r) = 1 - \frac{2s}{\abs{w}} r^{\alpha-1} (1-r^\alpha)^{\frac{2}{\alpha}-1}.$$
Clearly, $\phi_{\abs{w}, s}^{\prime}(r) \geq 1$ for all $r \in [0,1]$ when $s$ is negative. The difficulty 
arises when $s$ is positive. To handle this case, notice first that if $\alpha \geq 2$, then  
$r^{\alpha-1} (1-r^\alpha)^{\frac{2}{\alpha}-1}$ is a monotonically increasing function for 
$r \in[0,1)$, as follows from the fact that its derivative is strictly positive in $(0,1)$. 
Assume for now that $\alpha>2$, and we will take care of the case when $\alpha =2$ separately. 
When $\alpha > 2$ the limit of the latter function as $r \to 1^-$ is $+\infty$, and therefore 
there exists unique point $r_0 \equiv r_0(\abs{w},s)\in (0,1)$ at which 
$r_0^{\alpha-1} (1-r_0^\alpha)^{\frac{2}{\alpha}-1}  = \frac{\abs{w}}{4s}$.  
Then for $r \in [0, r_0]$, the first derivative is bounded from below by 
$$\abs{\phi_{\abs{w}, s}^{\prime}(r)} \geq \frac{1}{2} \, .$$
In the complementary range $[r_0,1)$, we will now show that $\abs{\phi_{\abs{w}, s}^{\prime \prime}(r)}$ 
is also bounded from below by $1/2$. Indeed, for $\alpha>2$,
$$\phi_{\abs{w}, s}^{\prime \prime}(r) = \frac{-2s}{\abs{w}} r^{\alpha-2} (1-r^\alpha)^{\frac{2}{\alpha}-2} \left((\alpha-1)-r^\alpha \right).$$
Notice that $\frac{ (\alpha-1) -r^{\alpha} }{ 1-r^{\alpha}} \ge 1$, and therefore for $r\in [r_0,1)$, 
it is easy to verify that
\begin{align*}
\abs{\phi_{\abs{w}, s}^{\prime \prime}(r)} &\geq \frac{2s}{\abs{w}} r^{\alpha-1} (1-r^\alpha)^{\frac{2}{\alpha}-1}\geq \frac{1}{2} \, .
\end{align*}
In summary, 
\begin{eqnarray*}
\abs{\phi_{\abs{w}, s}^{\prime}(r)} &\geq& \frac{1}{2} \, \text{ for } r \in [0, r_0],\\
\abs{\phi_{\abs{w}, s}^{\prime \prime}(r)} &\geq& \frac{1}{2} \, \text{ for } r \in [r_0, 1).
\end{eqnarray*}

So far we have seen that one can divide the interval $[0,1]$ into exactly two parts such that on 
one part $\abs{\phi_{\abs{w},s}^{\prime}}$ is bounded below by $1/2$, while on the other part 
$\abs{\phi_{\abs{w},s}^{\prime \prime}}$ is bounded below by $1/2$. In addition, 
$\phi_{\abs{w},s}^{\prime}$ is monotone on $(0,1)$, since $ \phi_{\abs{w},s}^{\prime \prime}$ 
has a constant sign there. We conclude that on each of such subintervals one can apply 
van-der Corput lemma  (in the form stated e.g. in  \cite[p. 332-334]{St}), so that :
$$\abs{w}^{-1/2} \abs{\int_0^1 \exp\left(i \abs{w} \phi_{\abs{w}, s}(r)\right) r^{d-\frac12} \, dr} 
\lesssim \abs{w}^{-1/2} \abs{w}^{-1/2} = \abs{w}^{-1}.$$

As to the case $\alpha=2$, let now $\abs{w}\phi_{\abs{w},s}(r) = \zeta(r) = \abs{w}r+s(1-r^2).$ 
Again we should consider just the case where $s$ is positive and $s\ge 1$, and then observe that if 
$\frac{\abs{w}}{4s} \geq 1$, then $\zeta^{\prime}(r) \geq \frac{\abs{w}}{2} \geq \frac{2}{\sqrt{17}} \abs{(w,s)}$ 
for all $r\in [0,1]$ (here $\abs{(w,s)}$ is the Euclidean norm). On the other hand, if $0<\frac{\abs{w}}{4s} \leq 1$, then 
$\zeta^{\prime \prime}(r) = 2s \geq \frac{2}{\sqrt{17}} \abs{(w,s)}$ 
for all $r\in [0,1]$. We now apply van-der Corput lemma on the interval $[0,1]$ and we have
$$\abs{w}^{-1/2} \abs{\int_0^1 e^{i \zeta(r)} r^{d-\frac12} \, dr} \lesssim \abs{w}^{-1/2} \abs{(w,s)}^{-1/2}.$$

Finally, comparing all the bounds leads to the choice of $\delta = \abs{w}^{-1}$ and this completes the proof of Lemma \ref{nonzero}.
\end{proof}


\section{Counting lattice point result in Heisenberg norm balls}
In the present section we will prove the following result :
\begin{thm}\label{mama-thm} For every $d\ge 1$, and $\alpha \ge 2$
\begin{align*}
\abs{ \# \left(\mathbb{Z}^{2d+1} \cap B_R^{\alpha,A}\right)  - \abs{B_R^{\alpha,A}}} \lesssim 
\begin{cases}
R^{2d} &; \mbox{ for } \alpha = 2\\
R^{2d} \log(R) &; \mbox{ for } 2<\alpha \leq 4\\
R^{2d+\delta} &; \mbox{ for } \alpha > 4 \, ,
\end{cases}
\end{align*}
where $\delta =\delta(\alpha) = \frac{2 \left(\frac12 -\frac{2}{\alpha}\right)}{d+ \frac12 -\frac{2}{\alpha}} \, .$ 
\end{thm}
We recall that the volume of $B_R^{\alpha,A}$ is given by $\abs{B_R^{\alpha,A}} = R^{2d+2} \abs{B_1^{\alpha,A}}$. 
Our proof will utilize the spectral estimates of the (Euclidean) Fourier transform of the unit ball $B_1^{\alpha,A}$ that were proved in the previous 
section, via the (Euclidean) Poisson summation formula, to which we now turn.

\subsection{Euclidean Poisson summation}
As noted in \S 2.1, we fix a bump function  $\rho: \mathbb{R}^{2d}\times \mathbb{R}\rightarrow \mathbb{R}$, which is a 
smooth non-negative function with support contained in the unit ball $B_1^{\alpha,A}$, such that $\rho(0) > 0$ and 
$\int_{B_1^{\alpha,A}} \rho(z,t) \, dz \, dt =1$. We then consider the family of functions defined by the Heisenberg dilates 
of $\rho$, namely $\rho_{\epsilon}(z,t) = \frac{1}{\epsilon^{2d+2}}\rho(\frac{z}{\epsilon}, \frac{t}{\epsilon^2})$.

\noindent {\bf Notation} : From here onward, to simplify the notation we write $B_R$ to denote $B_R^{{\alpha,A}}$.

We begin by obtaining estimates for $\# \left(\mathbb{Z}^{2d+1} \cap B_R\right)$ from above and below using Euclidean convolution.  
First, observe that by Proposition \ref{triangle}
\begin{align*}  \# \left(\mathbb{Z}^{2d+1} \cap B_R \right) &= \sum_{k\in \mathbb{Z}^{2d+1}} \chi_{B_R}(k) \\ 
& \le\sum_{k\in \mathbb{Z}^{2d+1}} \chi_{B_{R+\epsilon}}\ast \rho_{\epsilon}(k),
\end{align*}
where $\rho_\epsilon$ was defined above, and the convolution is Euclidean.  
Similarly, 
$$\# \left(\mathbb{Z}^{2d+1} \cap B_R \right) \geq \sum_{k\in \mathbb{Z}^{2d+1}} \chi_{B_{R-\epsilon}}* \rho_{\epsilon}(k).$$

We now use the Euclidean Poisson summation formula, so that the upper bound becomes :
\begin{align*}  
\# \left(\mathbb{Z}^{2d+1} \cap B_R\right) & \le \sum_{k\in \mathbb{Z}^{2d+1}} \left(\chi_{B_{R+\epsilon}} \ast 
\rho_{\epsilon}\right)^{\widehat{}}(k)\\
& = \sum_{k\in \mathbb{Z}^{2d+1} \setminus \{\vec{0}\}} \widehat{\chi_{B_{R+\epsilon}}}(k) \widehat{\rho_{\epsilon}}(k) + \abs{B_{R+\epsilon}} \, .
\end{align*}
Since $\rho$ is a compactly supported smooth bump function, the decay of the Fourier transform of 
$\rho_\epsilon(z,t)=\epsilon^{-2d-2}\rho(z/\epsilon,t/\epsilon^2)$  for any $N\in \N$ and any $(w,s)$ is bounded by 
\begin{equation}\label{rho-decay}
\abs{\widehat{\rho_\epsilon}(w,s)} = \abs{\widehat{\rho}(\epsilon w,\epsilon^2 s)}\le C_N 
\left(1+\abs{(\epsilon w,\epsilon^2 s)}\right)^{-N}\,.
\end{equation}

Continuing with the estimate, we have
\begin{align*} 
\# \left(\mathbb{Z}^{2d+1} \cap B_R \right) -\abs{B_R} 
&\le \sum_{k\in \mathbb{Z}^{2d+1} \setminus \{\vec{0}\}} \widehat{\chi_{B_{R+\epsilon}}}(k)\widehat{ \rho_{\epsilon}  }(k) 
+ \abs{B_{R+\epsilon}}-\abs{B_R}\\
&= \sum_{k\in \mathbb{Z}^{2d+1} \setminus \{\vec{0}\}} \widehat{\chi_{B_{R+\epsilon}}}(k)\widehat{ \rho_{\epsilon}  }(k) + 
O\left(R^{2d+1}\epsilon\right).
\end{align*}
Similarly,
$$\# \left(\mathbb{Z}^{2d+1} \cap B_R \right) - \abs{B_R} \geq - \abs{\sum_{k\in \mathbb{Z}^{2d+1} \setminus \{\vec{0}\}} 
\widehat{\chi_{B_{R-\epsilon}}}(k)\widehat{ \rho_{\epsilon}  }(k) } - O \left(R^{2d+1}\epsilon\right).$$
Combining these observations, we conclude that 
\begin{align*}\label{setup}
&\abs{\# \left(\mathbb{Z}^{2d+1} \cap B_R \right) -\abs{B_R}}\\
&\le \abs{\sum_{k\in \mathbb{Z}^{2d+1} \setminus \{\vec{0}\}} \widehat{\chi_{B_{R+\epsilon}}}(k)\widehat{ \rho_{\epsilon}}(k)} 
+ \abs{\sum_{\sum_{k\in \mathbb{Z}^{2d+1} \setminus \{\vec{0}\}}} \widehat{\chi_{B_{R-\epsilon}}}(k)\widehat{ \rho_{\epsilon}}(k)} 
+ O \left(R^{2d+1}\epsilon\right).
\end{align*}

\subsection{Estimates} \label{sec:estimate}

Our task is now to bound
\begin{equation}\label{mamasum}
R^{2d+2} \sum_{k\in \mathbb{Z}^{2d+1} \setminus \{\vec{0}\}} \abs{\widehat{\chi_{B_{1}}}(Rk',R^2k'')} 
\abs{\widehat{\rho}(\epsilon k',\epsilon^2 k'')} + O\left(R^{2d+1}\epsilon\right).
\end{equation}
%
%
%
We will break the sum further to three regions to utilize our three separate spectral decay estimates as follows:

\begin{align*}
& R^{2d+2}\sum \abs{\widehat{\chi_{B_{1}}}(Rk',R^2k'')} \abs{\widehat{\rho}(\epsilon k',\epsilon^2 k'')}+ O\left(R^{2d+1}\epsilon\right)\\
= & R^{2d+2} \left(\sum_{k''\neq 0} \abs{\widehat{\chi_{B_{1}}} \left(\vec{0},R^2k''\right)} \abs{\widehat{\rho}(\vec{0}, \epsilon^2 k'')}\right) 
 +R^{2d+2}\left(\sum_{k'\neq \vec{0}} \abs{\widehat{\chi_{B_{1}}} \left(Rk', 0\right)} \abs{\widehat{\rho}(\epsilon k', 0)}\right)\\
& + R^{2d+2}\left(\sum_{k'\neq \vec{0}, \, k''\neq 0} \abs{\widehat{\chi_{B_{1}}} \left(Rk',R^2k''\right)}\abs{\widehat{\rho}(\epsilon k',\epsilon^2 k'')} \right)+ O\left(R^{2d+1}\epsilon\right)\\
:= & I + II + III + O\left(R^{2d+1}\epsilon\right).
\end{align*}

For $I$, we use Lemma \ref{taxis},
\begin{align*}
I &= R^{2d+2}\sum_{k''\neq 0} \abs{\widehat{\chi_{B_{1}}} \left(\vec{0},R^2k''\right)}\abs{\widehat{\rho}(\vec{0},\epsilon^2 k'')}\\
&\lesssim R^{2d-\min\{\frac{\alpha}{2},\frac{2d}{\alpha} \}} \sum_{k''\neq 0} \abs{k''}^{-\left(1+ \min\{\frac{\alpha}{2},\frac{2d}{\alpha} \}\right)}\left(1+\epsilon^2 \abs{k''}\right)^{-N}\\
&\lesssim R^{2d-\min\{\frac{\alpha}{2},\frac{2d}{\alpha} \}}.
\end{align*}

In order to take advantage of the decay of $\widehat{\rho}$ (as stated in equation (\ref{rho-decay})) in estimating $II$ and $III$, 
we break the sum in $k= (k',k'')$ further into four pieces :\\
\textbf{sum 1 :} $\abs{k'} \leq \epsilon^{-1}$, $\abs{k''} \leq \epsilon^{-2}$, and $k \neq \vec{0}$ ;\\
\textbf{sum 2 :} $\abs{k'} \geq \epsilon^{-1}$, $\abs{k''} \leq \epsilon^{-2}$, and $k \neq \vec{0}$ ;\\
\textbf{sum 3 :} $\abs{k'} \leq \epsilon^{-1}$, $\abs{k''} \geq \epsilon^{-2}$, and $k \neq \vec{0}$ ;\\
\textbf{sum 4 :} $\abs{k'} \geq \epsilon^{-1}$, $\abs{k''} \geq \epsilon^{-2}$, and $k \neq \vec{0}$ .
For II, we use Lemma \ref{xyaxis}. We consider the case when $d-\frac12 \geq \frac{2}{\alpha}$ and $d-\frac12 \le \frac{2}{\alpha}$ separately.  
In either case, we need only consider $II$ over the set of $k\in \mathbb{Z}^{2d+1}$ satisfying the conditions listed in \textbf{sum 1} and \textbf{sum 2} with last coordinate equal to zero.

When $d-\frac12 \geq \frac{2}{\alpha}$,
\begin{align*}
II &= R^{2d+2} \sum_{k'\neq \vec{0}} \abs{\widehat{\chi_{B_{1}}} \left(Rk',0\right)} \abs{\widehat{\rho}(\epsilon k',0)}\\
&\lesssim R^{2d+2} \sum_{k'\neq \vec{0}} \abs{Rk'}^{-(d+\frac{1}{2}+ \frac{2}{\alpha})} \left(1+\epsilon \abs{k'}\right)^{-N}.
\end{align*}

We first sum over the conditions listed in \textbf{sum 1} to get
\begin{align*}
II_{sum 1} 
&\lesssim R^{2d+2} \sum_{1 \leq \abs{k'} \leq \epsilon^{-1}} \abs{Rk'}^{-(d+\frac{1}{2}+ \frac{2}{\alpha})} \left(1+\epsilon \abs{k'}\right)^{-N}.\\
&\lesssim R^{d+\frac{3}{2}-\frac{2}{\alpha}}\left( 1/ \epsilon \right)^{d-\frac{1}{2}-\frac{2}{\alpha}} \,.
\end{align*}

Next, summing over the conditions listed in \textbf{sum 2}, we get 
\begin{align*}
II_{sum 2} 
&\lesssim R^{2d+2} \sum_{\abs{k'} \geq \epsilon^{-1}} \abs{Rk'}^{-(d+\frac{1}{2}+ \frac{2}{\alpha})} \abs{\epsilon k'}^{-N}.\\
&\lesssim R^{d+\frac{3}{2}-\frac{2}{\alpha}}\left( 1/ \epsilon \right)^{d-\frac{1}{2}-\frac{2}{\alpha}} \,.
\end{align*}

We conclude that, for $d-\frac12 \geq \frac{2}{\alpha}$, $II \lesssim R^{d+\frac{3}{2}-\frac{2}{\alpha}}\left( \frac{1}{\epsilon} \right)^{d-\frac{1}{2}-\frac{2}{\alpha}}$.

When $d-\frac12 \le \frac{2}{\alpha}$, which is only the case for $d=1$ since $\alpha\geq 2$,
\begin{align*}
II &= R^{4} \sum_{k'\neq \vec{0}} \abs{\widehat{\chi_{B_{1}}}\left(Rk',0\right)} \abs{\widehat{\rho}(\epsilon k',0)}\\
&\lesssim R^{4} \sum_{k'\neq \vec{0}} \abs{Rk'}^{-2} \left(1+\epsilon \abs{k'}\right)^{-N}.
\end{align*}

Once again, we first sum over the conditions listed in \textbf{sum 1} to get
\begin{align*}
II_{sum 1} 
&\lesssim R^{4} \sum_{1 \leq \abs{k'} \leq \epsilon^{-1}} \abs{Rk'}^{-2}\\
&\sim R^2 \log\left(1 / \epsilon\right).
\end{align*}

And, summing over the conditions listed in \textbf{sum 2}, we get 
\begin{align*}
II_{sum 2} 
&\lesssim R^{4} \sum_{\abs{k'} \geq \epsilon^{-1}} \abs{Rk'}^{-2} \abs{\epsilon k'}^{-N}\\
&\sim R^2 \log\left(1 / \epsilon\right).
\end{align*}

We conclude that, for $d-\frac12 \le \frac{2}{\alpha}$, $II \lesssim R^2 \log\left(1 / \epsilon\right)$. 
We note for future use that for $d=1$ and $\alpha=2$  we have $II\lesssim R^{3/2}$, as follows from the fact 
that the decay of $\abs{\widehat{\chi_{B_1^{2,A}}} \left(Rk^\prime, 0\right)}$ is $R^{-5/2}$, as stated in Lemma \ref{xyaxis}. 

For III, we use Lemma \ref{nonzero} and \eqref{rho-decay}.  We begin by considering the case when $\alpha>2$.  We have 
\begin{align*}
III &= R^{2d+2} \sum_{k'\neq \vec{0}, \, k''\neq 0} \abs{\widehat{\chi_{B_{1}}}(Rk',R^2k'')}\abs{\widehat{\rho}(\epsilon k',\epsilon^2 k'')}\\
&\lesssim R^{d}\sum_{k'\neq \vec{0}, \, k''\neq 0} \abs{k'}^{-d} \abs{k''}^{-1}   \left(1+\abs{(\epsilon k',\epsilon^2 k'')}\right)^{-N}.
\end{align*}

Let us first consider $III$ with conditions listed in \textbf{sum 1}.
\begin{align*}
III_{sum 1} 
&\le R^{d}\sum_{sum 1} \abs{k'}^{-d} \abs{k''}^{-1}\\     
&\sim R^d \left(1/ \epsilon\right)^d \log\left(1 / \epsilon\right).
\end{align*}

Next, we consider $III$ with conditions listed in \textbf{sum 2}.
\begin{align*}
III_{sum 2} 
&\lesssim R^{d} \sum_{sum 2} \abs{k'}^{-d} \abs{k''}^{-1} \abs{\epsilon k'}^{-N}\\
&\lesssim R^d \left(1/ \epsilon\right)^d \log\left(1 / \epsilon\right).
\end{align*}

Now, we consider $III$ with conditions listed in \textbf{sum 3}.
\begin{align*}
III_{sum 3} &\lesssim R^{d}\sum_{sum 3} \abs{k'}^{-d} \abs{k''}^{-1}   \abs{\epsilon^2 k''}^{-N}\\
&\lesssim R^d \left(1/ \epsilon\right)^d  .
\end{align*}

Finally, we consider $III$ with conditions listed in \textbf{sum 4}.
\begin{align*}
III_{sum 4} & \lesssim R^d \sum_{sum 4} \abs{k'}^{-d} \abs{k''}^{-1}   \abs{\epsilon k'}^{-N/2}\abs{\epsilon^2 k''}^{-N/2}\\
&\lesssim R^d \left(1/ \epsilon\right)^d .
\end{align*}

We conclude that for $\alpha>2$, $III\lesssim  R^d \left(1/ \epsilon\right)^d \log\left(\frac{1}{\epsilon}\right)$.  

In the case $\alpha=2$ we have $III\lesssim R^{d-1/2} \left(1/ \epsilon\right)^{d+1/2}$ which 
follows from the previous calculation using the decay estimate 
$\abs{\widehat{\chi_{B_1}} \left(Rk^\prime, R^2 k''\right)} \lesssim \abs{Rk^\prime}^{-(d-1/2)} \abs{R^2k''}^{-3/2}$ 
stated in Lemma \ref{nonzero}.

We conclude that \eqref{mamasum} is bounded for $ \alpha > 2$ by 
\begin{align*}
&R^{2d-\min\left\{\frac{\alpha}{2}, \, \frac{2d}{\alpha}\right\}} 
+ R^{d+2-\min\left\{d, \, \frac{1}{2}+\frac{2}{\alpha}\right\}} \left(1 / \epsilon \right)^{d-\min\left\{d,\frac{1}{2}+\frac{2}{\alpha}\right\}} 
\log\left(1 / \epsilon \right)\\
& +  R^d \left(1 / \epsilon\right)^d \log\left(1 / \epsilon\right) + O\left(R^{2d+1}\epsilon\right),
\end{align*}
and for $\alpha =2$ the upper bound is $R^{d-\frac12}\left(1 / \epsilon\right)^{d+\frac12}$, as mentioned earlier in the proof, 
and appears without the logarithmic term. 

Next, we find the value of $\epsilon$ which minimizes the sum above. This is accomplished by graphing each of the four quantities 
in the sum above as functions of $\epsilon$ in order to find the value of epsilon which minimizes the maximum over the four quantities.  
For $2\le \alpha \le4$, we set $\epsilon = \frac{1}{R}$ and conclude that \eqref{mamasum} is bounded by $R^{2d}$ when $ \alpha = 2,$ 
and $R^{2d}\log(R)$ when $\alpha > 2$.

For $\alpha > 4$, we set $\epsilon = R^{\delta-1}$, 
where $\delta =\frac{2(\frac12 -\frac{2}{\alpha})}{d+ \frac12 -\frac{2}{\alpha}}$, 
and conclude that \eqref{mamasum} is bounded by $R^{2d+ \delta}\log(R).$  

This completes the proof of Theorem \ref{mama-thm}.

\section{Hyperplane slicing arguments} 
We now turn to the standard method of counting the lattice points in a given body by slicing 
it with hyperplanes, and counting the lattice points in each hyperplane separately, and establish what 
this method yields in the case under consideration. For notational simplicity, we write the arguments 
only for $A=1$, though the same estimates hold for general case with only non-essential modifications.

\subsection{Error bounds in higher dimensions} \label{slicing-d>1}

Slicing by hyperplane produces the following obvious identity : 
\begin{align*}
\# \left(\Z^{2d+1}\cap B_R^\alpha\right) &= \# \left(\set{(k', k'') \in \mathbb{Z}^{2d} \times \mathbb{Z} \,:\, \abs{k'}^\alpha + 
\abs{k''}^{\alpha/2} \le R^\alpha}\right)\\
&=\sum_{\abs{k''}\le R^2} \# \left({\set{k' \in \Z^{2d} \,:\, \abs{k'} \le \left(R^\alpha-\abs{k''}^{\alpha/2}\right)^{1/\alpha}}}\right)\\
&=\sum_{\abs{k''}\le R^2} G_{2d} \left(\left(R^\alpha - \abs{k''}^{\alpha/2}\right)^{1/\alpha}\right),
\end{align*}
where $G_{2d}(T)$ denotes $\# \left({\set{k' \in \Z^{2d} \,:\, \abs{k'} \le T}}\right),$ the standard lattice point counting function in 
Euclidean balls of radius $T$ in $\R^{2d}$. Let us write, with $A_{2d}$ denoting the volume of the unit ball in $\R^{2d}$ :
\begin{align*}
\abs{G_{2d}(T) - A_{2d}T^{2d}} \lesssim T^{\theta_1(2d)} \left(\log(2+T)\right)^{\theta_2(2d)}.
\end{align*}
From the known error estimates on counting lattice points in Euclidean balls in $\R^{2d}$ (see \cite{Kr1} for a full discussion), we have 
$\theta_1(2d) = 2d-2$ for all $d \geq 2$, $\theta_1(2)$ is conjectured to be $\frac12+\epsilon$ for any $\epsilon > 0$ (the conjectured 
error in the Gauss circle problem), $\theta_2(4) = 2/3$ while $\theta_2(2d) = 0$ for all $d \geq 3$. Thus,
$$\abs{\# \left(\Z^{2d+1}\cap B_R^\alpha\right) - \abs{B_R^\alpha}} \lesssim \abs{E_1 - \abs{B_R^\alpha}} + E_2 \, ,$$
where,
\begin{align*}
E_1 &= A_{2d} \sum_{\abs{k''}\le R^2} \left(R^\alpha-\abs{k''}^{\alpha/2}\right)^{2d/\alpha},\\
E_2 &= \sum_{\abs{k''}\le R^2} \left(\left(R^\alpha-\abs{k''}^{\alpha/2}\right)^{1/\alpha}\right)^{\theta_1(2d)} 
\left(\log \left(2+ \left(R^\alpha-\abs{k''}^{\alpha/2}\right)^{1/\alpha}\right)\right)^{\theta_2(2d)}.
\end{align*}
We will first show that the main term given by the volume satisfies  for $d \ge 1$ :


$$\abs{E_1 - \abs{B_R^\alpha}} \lesssim R^{2d- \min\left\{2,\frac{4d}{\alpha}\right\}}.$$
For this, first notice that

\begin{align*}
\abs{B_R^\alpha} 
= \int_{-R^2}^{R^2} \left(\int_{\abs{z}\le \left(R^\alpha-\abs{t}^{\alpha/2}\right)^{1/\alpha}} \, dz\right) \, dt
= A_{2d} \int_{-R^2}^{R^2} (R^\alpha-\abs{t}^{\alpha/2})^{2d/\alpha} \, dt,
\end{align*}
and therefore $E_1-\abs{B_R^\alpha}$ equals (up to a constant multiple) to
\begin{align*}
& \sum_{\abs{k''}\le R^2} \left(R^\alpha-\abs{k''}^{\alpha/2}\right)^{2d/\alpha} - 
\int_{-R^2}^{R^2} \left(R^\alpha-\abs{t}^{\alpha/2}\right)^{2d/\alpha} \, dt\\
=& R^{2d} \sum_{-R^2 \le k'' \le R^2} g\left(k''/R^2\right) - R^{2d} \int_{-R^2}^{R^2} g\left(t/R^2\right) \, dt \, ,
\end{align*}
where $g(t) = \left(1-\abs{t}^{\alpha/2}\right)^{2d/\alpha}$ on $[-1,1]$. 

Now we apply the Euler-MacLaurin formula (as in \S 3 \cite{KN1}, see also \cite{Kr1}, p.20-23) which says that
\begin{align*}
\sum_{-R^2 \le k''\le R^2} g\left(k''/R^2\right) - \int_{-R^2}^{R^2} g\left(t/R^2\right) \, dt 
&= \int_{-R^2}^{R^2} \psi(t) \frac{d}{dt} \left(g\left(t/R^2\right)\right) \, dt \\
&= 2 \int_0^1 \psi\left(R^2t\right) g^{\prime}(t) \, dt \, ,
\end{align*}
with $\psi(t) = t - [t] - \frac12$. Making use of the Fourier series expansion of $\psi$, namely 
$\psi(t) = -\pi^{-1} \sum_{j=1}^{\infty} j^{-1} \sin(2\pi jt)$, the above estimate equals to
\begin{align*}
\frac{2d}{\pi} \sum_{j=1}^{\infty} \frac{1}{j} \int_0^1 t^{\frac{\alpha}{2}-1} 
\left(1-t^{\alpha/2}\right)^{\frac{2d}{\alpha}-1} \sin\left(2\pi j R^2 t\right) \, dt.
\end{align*}
To estimate each integral of the above summand, we perform one integration by parts if $\frac{2d}{\alpha} \geq 1$, otherwise we use the 
estimates for oscillatory integrals with singularities at an endpoint as stated in \S \ref{vdC}, to conclude that
\begin{align*}
\abs{\int_0^1 t^{\frac{\alpha}{2}-1} \left(1-t^{\alpha/2}\right)^{\frac{2d}{\alpha}-1} \sin\left(2\pi j R^2 t\right) \, dt} 
\lesssim\left(R^2 j\right)^{-\min\left\{1,\frac{2d}{\alpha}\right\}}.
\end{align*}
Collecting all these estimates, we get the claimed bound on $\abs{E_1 - \abs{B_R^\alpha}}$.

Now we perform the following calculation to estime $E_2$ : When $d\geq2$, we get
\begin{align}
\nonumber E_2 &= \sum_{\abs{k''}\le R^2} \left(\left(R^\alpha-\abs{k''}^{\alpha/2}\right)^{1/\alpha}\right)^{\theta_1(2d)} 
\left(\log \left(2+ \left(R^\alpha-\abs{k''}^{\alpha/2}\right)^{1/\alpha}\right)\right)^{\theta_2(2d)}\\
&\sim R^{2+{\theta_1(2d)}} \left(\log(R)\right)^{\theta_2(2d)}. \label{E_2}
\end{align}

Combining the above estimates we see that
\begin{align}
\abs{\# \left(\Z^{2d+1}\cap B_R^\alpha\right) - \abs{B_R^\alpha}} \lesssim
\begin{cases}
R^4 (\log(R))^{2/3} &; \mbox{ for } d=2\\
R^{2d} &; \mbox{ for } d \geq 3 \,.\\
\end{cases}
\end{align}

\subsection{Error bounds in the first Heisenebrg group ${\sf H}_1$} \label{slicing-d=1}
In the case $d=1$ of the first Heisenebrg group ${\sf H}_1$, we can perform the same calculation as in the previous section, but must 
decide on the value we take for $\theta_1(2)$. Let us note that even if we assume the best possible conjectured error estimate in the 
Gauss circle problem, namely $\theta_1(2)=\frac12 + \epsilon$, then \eqref{E_2} implies
$$\abs{E_2} = O_\epsilon \left(R^{\frac{5}{2} + \epsilon}\right).$$
Therefore,
\begin{align}
\abs{\# \left(\Z^3 \cap B_R^\alpha\right) - \abs{B_R^\alpha}} = O_\epsilon \left(R^{\frac{5}{2} + \epsilon}\right)
\end{align}
for every $\epsilon > 0$. However, as stated in Theorem \ref{mama-thm}, the estimate we obtain for example for $2 \le \alpha \le 4$ is 
at most $R^2 \log(R)$, which is better, see below for further discussion.

%


\section{Proof of the main theorem} \label{sec:proof-main}
We will now put everything together and prove Theorem \ref{thm:main}. 

\subsection{The best possible estimate for $\alpha =2$ in all dimensions}
The fact that
$$\abs{\# \left(\mathbb{Z}^{2d+1} \cap B_R^{2,A}\right) - \abs{B^{2,A}_{R}}} = O(R^{2d})$$
follows from the corresponding case in Theorem \ref{mama-thm}. The matching lower bound for the error, namely the fact that 
$$\abs{\# \left(\mathbb{Z}^{2d+1} \cap B_R^{2,A}\right) - \abs{B^{2,A}_{R}}} = \Omega(R^{2d})$$
follows from the following elementary argument. Assume to the contrary that
$$\abs{\# \left(\mathbb{Z}^{2d+1} \cap B_R^{2,A}\right) - \abs{B^{2,A}_{R}}} = o(R^{2d}).$$
Then for any $M \in \mathbb{N}$, using the fact that we are counting integer points 
\begin{align*}
0 &= \# \left(\mathbb{Z}^{2d+1} \cap B_{\sqrt{M+\frac{1}{2}}}^{2,A} \right) - \# \left(\mathbb{Z}^{2d+1} \cap B_{\sqrt{M}}^{2,A}\right)\\
&= \abs{B^{2,A}_{\sqrt{M+\frac{1}{2}}}} - \abs{B^{2,A}_{\sqrt{M}}} + o(M^d) \\
&= C \, M^d + o(M^d),
\end{align*}
which is a contradiction.

\subsection{The first Heisenebrg group ${\sf H}_1$}

\subsubsection{The case $  \alpha  > 2 $.}
 Theorem \ref{mama-thm} implies that  
\begin{align*}
\abs{\# \left(\mathbb{Z}^{3} \cap B_R^{\alpha,A}\right) - \abs{B^{\alpha,A}_{R}}} \le 
\begin{cases}
R^{2} \log(R) &; \mbox{ for } 2<\alpha \leq 4\\ 
R^{2+\delta} &; \mbox{ for } \alpha > 4\,,
\end{cases}
\end{align*}
where $\delta =\delta(\alpha)=\frac{2(\frac12 -\frac{2}{\alpha})}{\frac32 -\frac{2}{\alpha}}$. 

Note that the bound on the error term for the Cygan-Kor\'anyi norm ($\alpha=4$) for example is the same (up to a $\log$ factor) as in 
the case of $\alpha=2$ which was best possible. However, we do not have any meaningful $\Omega$-result here, and the upper bound on 
the error term can most likely be improved in this case.

\subsubsection{Slicing and the error bound for large $\alpha$.}

As soon as $\alpha > 4$, we have $\delta(\alpha) > 0$ and so the quality of the error term obtained in Theorem \ref{mama-thm} declines 
as $\alpha \to \infty$. As we saw at the end of the previous section, in the case of ${\sf H}_1$ the slicing argument produces different 
results, depending on the quality of the error estimate in the classical Gauss circle problem. Using the best possible conjectured 
estimate in the latter problem produces an estimate of the error term in 
$\abs{\# \left(\mathbb{Z}^{3} \cap B_R^{\alpha,A}\right) - \abs{B^{\alpha,A}_{R}}}$ 
which is $R^{\frac{5}{2} + \epsilon}$, so that it is inferior to the one stated in Theorem \ref{mama-thm}, namely $R^2 \log(R)$ when 
$2 < \alpha \le 4$. However, for sufficiently large $\alpha$, the error estimate obtained by slicing, namely $R^{2+\theta_1(2)}$ is superior 
to the one mentioned in Theorem \ref{thm:main}, namely $R^{2+\delta(\alpha)}$. Indeed, $\lim_{\alpha \to \infty} \delta(\alpha)=2/3$, 
so for any value of $\theta_1(2)$ which is less than $2/3$, there exists $\alpha_\theta$ such that for $\alpha > \alpha_\theta$ the estimate 
produced by slicing is better than that stated in Theorem \ref{mama-thm}. For example, when we choose the conjectured best possible value 
$\theta_1(2)=1/2$, we have $\alpha_{1/2}=12.$ In particular, for $2\le \alpha \le 12$, the estimate stated in Theorem \ref{mama-thm} is 
superior to the one produced by slicing. 

\subsection{The case of ${\sf H}_d$, $d \ge 2$}
For the higher-dimensional Heisenberg groups, the method of slicing carried out in \S \ref{slicing-d>1} is surprisingly effective 
and produces the  following bound on the error term 
\begin{align*}
\abs{\# \left(\mathbb{Z}^{2d+1} \cap B_R^{\alpha,A}\right) - \abs{B^{\alpha,A}_{R}}} \le 
\begin{cases}
R^{4} \log(R) &; \mbox{ for } d=2\\ 
R^{2d} &; \mbox{ for } d \ge 3\,.
\end{cases}
\end{align*}
For $2 < \alpha\le 4$ this is the same bound (up to a $\log$ factor)  that was obtained in Theorem \ref{mama-thm}. However, note that in 
the slicing argument we utilized the best possible result for the classical lattice point counting problem in Euclidean balls in dimensions 
$2d\ge 4$. This is a highly non-elementary result, and it is interesting to note that Theorem \ref{mama-thm} produces the bound 
$R^{2d}\log(R)$ for $2 < \alpha \le 4 $ just using van-der-Corput lemma and Poisson summation. 

Another interesting comment is that convexity of $B_R^{\alpha,A}$ is irrelevant to the slicing argument, and in fact the error estimates 
stated in the beginning of this subsection are valid for all $\alpha > 0$, but $B_1^{\alpha,A}$  is convex if and only if $\alpha \ge 2$.   

This concludes the proof of Theorem \ref{thm:main}.

\section{Comparison with some Euclidean lattice point counting results}

In this section we consider the problem of applying an analogue of our method to the problem of lattice point counting in the 
{\it Euclidean dilates} of the compact, convex 3-dimensional Euclidean bodies
$$D_1^{\alpha} = \{(z,t)\in \mathbb{R}^2\times \mathbb{R} \,:\, \widetilde{N}_\alpha ((z,t)) \leq 1\}.$$
where $\widetilde{N}_\alpha ((z,t)) = \left(\abs{z}^\alpha + \abs{t}^\alpha\right)^{1/\alpha}$. Namely we will estimate
$$\abs{\# \left(\mathbb{Z}^{3} \cap D_R^{\alpha} \right) - \abs{D_R^{\alpha}}}$$
where $D_R^{\alpha} = \{(Rz,Rt) \,:\, (z,t)\in D_1^{\alpha}\}$ so that
$\abs{D_R^{\alpha}} = R^3 \abs{D_1^{\alpha}}$.

This problem was studied in great detail in \cite{KN1} (see also the references therein) where the authors performed very fine analysis 
obtaining sharp asymptotic results. Our goal in the present section is to demonstrate that our method of utilizing direct spectral decay 
estimates via Poisson's summation formula also gives the right {\it first order} error estimate (up to a $\log$ factor) for $\alpha\ge 4$, 
equal to the one produced in \cite{KN1}. Our method cannot reproduce the much finer results regarding the secondary main terms obtained 
in \cite{KN1}. The reason we include this analysis here is in order to establish some point of comparison with which to gauge the quality 
of the error estimate stated in Theorem \ref{thm:main}. As noted in the introduction, the authors are not aware of any previous result on 
the lattice point counting problem on the Heisenberg groups which could serve as a basis for such a comparison. 

Our main result in the Euclidean setting is as follows.

\begin{thm}\label{generalmain}
For $\alpha > 2$,
$$\abs{\# \left(\mathbb{Z}^{3} \cap D_R^{\alpha}\right) - \abs{D_R^{\alpha}}} = O\left(\left(R^{\frac{3}{2}} + 
R^{2-\frac{2}{\alpha}}\right)\log(R) \right) \,.$$
\end{thm}

In the next section, we state spectral decay estimates analogous to those of \S \ref{Heisenberg-spectral}, which we will use to  prove 
Theorem \ref{generalmain} in \S \ref{proof-alpha-alpha}.

\subsection{Spectral decay estimates}
We begin with a remark that though for simplicity we have stated Theorem \ref{generalmain} for $d=1$, our result extends to higher dimensions, 
namely $\R^{2d+1}$ as well. We restrict our attention only to the case $\alpha >2$. The estimates in this section are stated for $d\geq 1$ 
and $\alpha > 2$. We start with the analogs of Lemmas \ref{taxis}, \ref{xyaxis} and \ref{nonzero}. Each of the three lemmas stated in this 
section are proved using the same techniques as above. Therefore, rather than repeat each of the proofs above, we include the statement of 
the key estimates with some brief comments regarding the proof. 

\begin{lem} \label{generaltaxis}
For $\alpha > 2$ , $\abs{s} \geq 1,$
$$\abs{\widehat{\chi_{D_1^{\alpha}}}(\vec{0},s)} \lesssim \abs{s}^{-\left(1+\min \left\{ \frac{2d}{\alpha}, \, \alpha \right\}\right)}.$$
However, for a positive integer $\alpha \in 2 \N$, we have a better estimate
$$\abs{\widehat{\chi_{D_1^{\alpha}}}(\vec{0},s)} \lesssim \abs{s}^{-\left(1+\frac{2d}{\alpha}\right)}.$$
\end{lem}
For the purpose of proving Theorem \ref{generalmain} where $d=1$, note that $\min \left\{ \frac{2}{\alpha}, \alpha \right\} = \frac{2}{\alpha}$.
The proof of lemma \ref{generaltaxis} is similar to that of Lemma \ref{taxis}, and so we do not repeat it here.

\begin{lem} \label{generalxyaxis}
For $\alpha > 2$, $\abs{w} \geq 1$,
\begin{align*}
\abs{\widehat{\chi_{D_1^{\alpha}}}(w,0)} &\lesssim 
\begin{cases}
\abs{w}^{-2d} &; \mbox{ if } \frac{1}{\alpha} > d-\frac{1}{2}\\
\abs{w}^{-\left(d+\frac{1}{2}+\frac{1}{\alpha} \right)} &; \mbox{ if } \frac{1}{\alpha} \leq d-\frac{1}{2}.
\end{cases}
\end{align*}
\end{lem}
For the purpose of proving Theorem \ref{generalmain} where $d=1$, we will only need to use the second estimate.

\begin{proof}
The conclusion of the first item agrees with the first item in Lemma \ref{xyaxis}, and the proof follows similarly.
In the case that $\frac{1}{\alpha} \leq d-\frac{1}{2}$, we follow the proof of Lemma \ref{xyaxis} to see that for $k=0,1$ :
$$\abs{\int_0^1 e^{i \xi r} r^{d-k-\frac{1}{2}} (1-r^\alpha)^{1/\alpha} \, dr} \lesssim 
\xi^{-\left(1+ \min \left\{\alpha, \, \frac{1}{\alpha}, \, d-k-\frac{1}{2}\right\}\right)}.$$
Collecting all the bounds, we see that $I(\xi) = \int_0^1 J_{d-1}(\xi r) \left(1-r^\alpha \right)^{2/\alpha} r^d \, dr$ 
is dominated by a finite sum of terms of the form
$$\delta^{d+1}, \, \xi^{-\frac52} \delta^{-\frac{1}{2}}, \, \xi^{-\left(k+\frac{1}{2}\right)} \delta^{d-k+\frac{1}{2}}, \, 
\xi^{-\left(k+\frac{3}{2}\right)} \xi^{-\min \left\{\alpha, \, \frac{1}{\alpha}, \, d-k-\frac{1}{2}\right\}}$$
where $k=0,1.$ Choosing $\delta = \xi^{-1}$, we verify that
%
\begin{eqnarray*}
\absolute{I(\xi)} \lesssim \xi^{-\left(\frac{3}{2}+\frac{1}{\alpha}\right)} \, ; \, \textup{ whenever } \frac{1}{\alpha} \leq d-\frac{1}{2}.
\end{eqnarray*}
\end{proof}

\begin{lem}\label{generalnonzero}
For $\alpha > 2$ we have for $\abs{w} \geq 1$ and $\abs{s}\geq 1$,
\begin{eqnarray}
\abs{\widehat{\chi_{D_1^{\alpha}}}(w,s)} \lesssim  \abs{w}^{-d} \abs{s}^{-1}.   
\end{eqnarray}
\end{lem}

\begin{proof} Considering the proof of Lemma \ref{nonzero}, we see that it suffices to estimate
$$\abs{w}^{-1/2} \int_0^1 e^{i\abs{w} r}e^{is (1-r^\alpha)^{1/\alpha}} r^{d-\frac12} \, dr \, .$$
The above integral is same as
$$\abs{w}^{-1/2} \int_0^1 \exp \left(i\abs{w} \phi_{\abs{w}, s}(r)\right) r^{d-\frac12} \, dr \, ,$$
where $\phi_{\abs{w}, s}(r) = r + \frac{s}{\abs{w}} (1-r^\alpha)^{1/\alpha}.$ Now
\begin{eqnarray*}
\phi_{\abs{w}, s}'(r) &=& 1 -  \frac{s}{\abs{w}} r^{\alpha-1} (1-r^\alpha)^{\frac{1}{\alpha}-1} \, , \\
\phi_{\abs{w}, s}''(r) &=& - (\alpha -1) \frac{s}{\abs{w}} r^{\alpha-2} (1-r^\alpha)^{\frac{1}{\alpha}-2} \, .
\end{eqnarray*}
Therefore, $\phi_{\abs{w}, s}'$ is always a monotic function. Clearly, $\phi_{\abs{w}, s}'(r) \geq 1$ for all $r \in [0,1]$
when $s$ is negative. The difficulty arises only when $s$ is positive.

Since $r^{\alpha-1} (1-r^\alpha)^{\frac{1}{\alpha}-1}$ is a strictly increasing function mapping 
$[0,1)$ onto $[0,\infty)$, there exists unique point $r_0 \equiv r_0(\abs{w},s)\in (0,1)$ at which
$r_0^{\alpha-1} (1-r_0^\alpha)^{\frac{1}{\alpha}-1} = \frac{\abs{w}}{2s}$ and
\begin{eqnarray*}
\abs{\phi_{\abs{w}, s}'(r)} \geq &\frac{1}{2} \, \mbox{ for all } r \in [0, r_0],\\
\abs{\phi_{\abs{w}, s}''(r)} \geq &\frac{\alpha - 1}{2} \, \mbox{ for all } r \in [r_0,1).
\end{eqnarray*}
The rest of the proof is analogous to that of Lemma \ref{nonzero}.
\end{proof}

\subsection{Proof of Theorem \ref{generalmain}} \label{proof-alpha-alpha}
We fix a bump function  $\rho: \mathbb{R}^{2}\times \mathbb{R}\rightarrow \mathbb{R}$, which is a smooth non-negative function with 
support contained in the unit ball $D_1^{\alpha}$, such that $\rho(0) > 0$ and $\int_{D_1^{\alpha}} \rho(z,t) \, dz \, dt =1$. 
We then consider the family of functions defined by the Euclidean dilates of $\rho$, namely 
$\rho_\epsilon (z,t) = \epsilon^{-3} \rho(\frac{z}{\epsilon}, \frac{t}{\epsilon})$.
%
%

In order to prove Theorem \ref{generalmain}, we will follow the proof of the Theorem \ref{mama-thm}. 
Rather than repeat the proof above, we include key estimates where the two arguments differ. 
As in \S \ref{sec:estimate}, the proof reduces to estimating :

\begin{equation}\label{generalmamasum}
R^{3}\sum_{k\in \mathbb{Z}^{3} \setminus \{\vec{0}\}} \abs{\widehat{\chi_{D_1^{\alpha}}}(Rk',Rk'')} 
\abs{\widehat{\rho}(\epsilon k',\epsilon k'')}+O\left(R^2 \epsilon \right).
\end{equation}

In order to utilize the spectral decay estimates of Lemmas \ref{generaltaxis}, \ref{generalxyaxis} 
and \ref{generalnonzero}, we break the sum in $k= (k',k'')\in \mathbb{Z}^{2}\times \mathbb{Z}$ 
into following four pieces :\\
\textbf{sum 1 :} $\abs{k'} \leq \epsilon^{-1}$, $\abs{k''} \leq \epsilon^{-1}$, and $k\neq \vec{0}$ ;\\
\textbf{sum 2 :} $\abs{k'} \geq \epsilon^{-1}$, $\abs{k''} \leq \epsilon^{-1}$, and $k\neq \vec{0}$ ;\\
\textbf{sum 3 :} $\abs{k'} \leq \epsilon^{-1}$, $\abs{k''} \geq \epsilon^{-1}$, and $k\neq \vec{0}$ ;\\
\textbf{sum 4 :} $\abs{k'} \geq \epsilon^{-1}$, $\abs{k''} \geq \epsilon^{-1}$, and $k\neq \vec{0}$ .

\vspace{1em} 

%

Continuing the analysis analogous to that of \S \ref{sec:estimate}, we conclude that \eqref{generalmamasum} is bounded 
by a constant times : 

\begin{equation}\label{generalmainevent} 
R^{2-\frac{2}{\alpha}}  +  R^{\frac32 -\frac{1}{\alpha}} \left(1 / \epsilon\right)^{\frac{1}{2} - \frac{1}{\alpha}} + 
\left(R / \epsilon\right)\log(R) + O\left(R^2\epsilon\right).
\end{equation}
When $2 < \alpha\le 4$, we choose $\epsilon = R^{-1/2}$ to see that the expression in \eqref{generalmainevent} is bounded by
$$R^{3/2}\log(R).$$
When $\alpha > 4$, we choose $\epsilon = R^{\frac{2}{\alpha}-1}$ to see that \eqref{generalmainevent} is bounded by
$$R^{2-\frac{2}{\alpha}}\log(R).$$  This completes the  proof of Theorem \ref{generalmain}.


\begin{thebibliography}{100}

\bibitem[BrD]{BrD} E. Breuillard and E. Le Donne, \textit{On the rate of convergence to the asymptotic cone for nilpotent groups and subFinsler geometry}, PNAS, 2012.

\bibitem[Ch]{Ch} F. Chamizo, \textit{Lattice points in bodies of revolution}, Acta Arith. \textbf{85(3)} (1998), 265-277.

\bibitem[Co]{Co} M. Cowling, \textit{Unitary and uniformly bounded representations of some simple Lie groups.
In: Harmonic Analysis and Group Representations}, C.I.M.E. Napoli; Liguori, 1982, 49-128.

\bibitem[CDKR]{CDKR} M. Cowling, A. H. Dooley, A. Kor\'anyi and F. Ricci,  \textit{$H$-type groups and Iwasawa decompositions}, Adv. Math. \textbf{87} (1991), 1-41.

\bibitem[Cy]{Cy} J. Cygan, \textit{Subadditivity of homogeneous norms on certain nilpotent Lie groups}, Proc. Amer. Math. Soc. \textbf{83(1)} (1981), 69-70.

\bibitem[Cy1]{Cy1} J. Cygan, \textit{Wiener's test for the Brownian motion on the Heisenberg group}, Colloq. Math. \textbf{39(2)} (1978), 367-373.

\bibitem[DM]{DM} M. Duchin and C. Mooney, \textit{Fine asymptotic geometry in the Heisenberg group}, Math. arXiv:1106.5276, July 2011.

\bibitem[Er]{Er} A. Erd\'{e}lyi, \textit{Asymptotic expansions}, Dover Publications, New York, 1956.

\bibitem[FL]{FL} R. L. Frank and E. H. Lieb, \textit{Sharp constants in several inequalities on the Heisenberg group}, Math. arXiv:1009.1410, Nov. 2011.

\bibitem[GN]{GN} A. Gorodnik and A. Nevo, \textit{Counting lattice points}, J. Reine Angew. Math. \textbf{663} (2012), 127-176.

\bibitem[Gr]{Gr} L. Grafakos, \textit{Classical Fourier analysis}, Second edition, GTM, \textbf{249}, Springer, New York, 2008.

\bibitem[He]{He} C. S. Herz, \textit{On the number of lattice points in a convex set}, Amer. J. Math. \textbf{84} (1962), 126-133.

\bibitem[Hl]{Hl} E. Hlawka, \textit{\"{U}ber Integrale auf konvexen K\"{o}rpern. I}, Monatsh. Math. \textbf{54} (1950), 1-36. 

\bibitem[IKKN]{IKKN} A. Ivi\'c, E. Kr\"atzel, M. K\"uhleitner and W.G. Nowak, \textit{Lattice points in large regions and related arithmetic functions: Recent developments in a very classic topic}, 
Elementare und analytische Zahlentheorie, Schr. Wiss. Ges. Johann Wolfgang Goethe Univ., Frankfurt am Main, \textbf{20} (2006), 89-128.

\bibitem[KP]{KP} M. Khosravi and Y. Petridis, \textit{The remainder in Weyl's law for $n$-dimensional Heisenberg manifolds},  Proc. Amer. Math. Soc. \textbf{133(12)} (2005), 3561-3571.

\bibitem[Kor]{Kor} A. Kor\'anyi, \textit{Geometric properties of Heisenberg-type groups}, Adv. Math. \textbf{56(1)} (1985), 28-38.

\bibitem[Kr1]{Kr1} E. Kr\"atzel, \textit{Lattice points},  Kluwer, Dordrecht-Boston-London, 1988.

\bibitem[Kr2]{Kr2} E. Kr\"atzel, \textit{Lattice points in super spheres}, Comment. Math. Univ. Carolin. \textbf{40(2)} (1999), 373-391.

\bibitem[Kr3]{Kr3} E. Kr\"atzel, \textit{Lattice points in some special three-dimensional convex bodies with points of gaussian curvature zero at the boundary}, Comment. Math. Univ. Carolin. \textbf{43(4)} (2002), 755-771.

\bibitem[Kr4]{Kr4} E. Kr\"atzel, \textit{Lattice points in three-dimensional convex bodies with points of Gaussian curvature zero at the boundary}, Monatsh. Math. \textbf{137(3)} (2002), 197-211.

\bibitem[KN1]{KN1} E. Kr\"atzel and W. G. Nowak, \textit{The lattice discrepancy of bodies bounded by a rotating Lam\'e's curve}, Monatsh. Math \textbf{154(2)} (2008), 145-156.

\bibitem[KN2]{KN2} E. Kr\"atzel and W. G. Nowak, \textit{The lattice discrepancy of certain three-dimensional bodies}, Monatsh. Math \textbf{163(2)} (2011), 149-174.

\bibitem[No]{No} W. G. Nowak, \textit{On the lattice discrepancy of bodies of rotation with boundary points of curvature zero}, Arch. Math. \textbf{90(2)} (2008), 181-192.

\bibitem[PP]{PP} J. Parkkonen and F. Paulin, \textit{Counting and equidistribution in Heisenberg groups}, Math. arXiv:1402.7225, Feb. 2014.

\bibitem[Pe1]{Pe1} M. Peter, \textit{The local contribution of zeros of curvature to lattice points asymptotics}, Math. Z. \textbf{233(4)} (2000), 803-815.

\bibitem[Pe2]{Pe2} M. Peter, \textit{Lattice points in convex bodies with planar points on the boundary}, Monatsh. Math. \textbf{135(1)} (2002), 37-57.

\bibitem[R1]{R1} B. Randol, \textit{A lattice point problem}, Trans. Amer. Math. Soc. \textbf{121} (1966), 257-268.

\bibitem[R2]{R2} B. Randol, \textit{A lattice point problem II}, Trans. Amer. Math. Soc. \textbf{125} (1966), 101-113.

\bibitem[R3]{R3} B. Randol, \textit{On the Fourier transform of the indicator function of a planar set}, Trans. Amer. Math. Soc. \textbf{139} (1969), 271-278.

\bibitem[R4]{R4} B. Randol, \textit{On the asymptotic behavior of the Fourier transform of the indicator function of a convex set}, Trans. Amer. Math. Soc. \textbf{139} (1969), 279-285.

\bibitem[St]{St} E. M. Stein, \textit{Harmonic analysis: real-variable methods, orthogonality, and oscillatory integrals}, Princeton University Press, 1993.

\bibitem[SW]{SW} E. M. Stein and S. Wainger, \textit{Problems in harmonic analysis related to curvature}, Bull. Amer. Math. Soc. \textbf{84(6)} (1978), 1239-1295.
\end{thebibliography}
\end{document}